\pdfoutput=1
\RequirePackage{ifpdf}
\ifpdf 
\documentclass[pdftex]{sigma}
\else
\documentclass{sigma}
\fi

\numberwithin{equation}{section}

\newtheorem{Theorem}{Theorem}[section]
\newtheorem{Corollary}[Theorem]{Corollary}
\newtheorem{Lemma}[Theorem]{Lemma}
\newtheorem{Proposition}[Theorem]{Proposition}

\newcommand{\der}{{\rm d}}

\begin{document}


\newcommand{\arXivNumber}{1506.02473}

\renewcommand{\PaperNumber}{029}

\FirstPageHeading

\ShortArticleName{Flat $(2,3,5)$-Distributions and Chazy's Equations}

\ArticleName{Flat $\boldsymbol{(2,3,5)}$-Distributions and Chazy's Equations}

\Author{Matthew RANDALL}

\AuthorNameForHeading{M.~Randall}

\Address{Department of Mathematics and Statistics, Faculty of Science, Masaryk University,\\ Kotl\'a\v{r}sk\'a 2, 611 37 Brno, Czech Republic}
\Email{\href{mailto:randallm@math.muni.cz}{randallm@math.muni.cz}}

\ArticleDates{Received September 23, 2015, in f\/inal form March 14, 2016; Published online March 18, 2016}

\Abstract{In the geometry of generic 2-plane f\/ields on 5-manifolds, the local equivalence problem was solved by Cartan who also constructed the fundamental curvature invariant. For generic 2-plane f\/ields or $(2,3,5)$-distributions determined by a single function of the form~$F(q)$, the vanishing condition for the curvature invariant is given by a 6$^{\rm th}$ order non\-linear ODE. Furthermore, An and Nurowski showed that this ODE is the Legendre transform of the 7$^{\rm th}$ order nonlinear ODE described in Dunajski and Sokolov. We show that the 6$^{\rm th}$ order ODE can be reduced to a 3$^{\rm rd}$ order nonlinear ODE that is a~generalised Chazy equation. The 7$^{\rm th}$ order ODE can similarly be reduced to another generalised Chazy equation, which has its Chazy parameter given by the reciprocal of the former. As a consequence of solving the related generalised Chazy equations, we obtain additional examples of f\/lat $(2,3,5)$-distributions not of the form $F(q)=q^m$. We also give 4-dimensional split signature metrics where their twistor distributions via the An--Nurowski construction have split~$G_2$ as their group of symmetries.}

\Keywords{generic rank two distribution in dimension f\/ive; conformal geometry; Chazy's equations}

\Classification{58A30; 53A30; 34A05; 34A34}

\section{Introduction}

The following $6^{\rm th}$ order nonlinear ODE
\begin{gather}\label{6thode}
10 F^{(6)}{F''}^3-80{F''}^2F^{(3)} F^{(5)}-51{F''}^2{F^{(4)}}^2+336 F'' {F^{(3)}}^2F^{(4)}-224{F^{(3)}}^4=0,
\end{gather}
arises in~\cite[Corollary~2.1]{annur} in the study of generic 2-plane f\/ields on 5-manifolds. The generi\-ci\-ty condition here means $F''(q) \neq 0$ in~(\ref{6thode}). This ODE arises as the integrability condition for generic 2-plane f\/ields on 5 manifolds determined by a~function of a single variable of the form~$F(q)$. A~generic 2-plane f\/ield~$\mathcal{D}$ on a~5-manifold~$M$ is a maximally non-integrable rank~2 distribution. For further details, see~\cite{annur, new, tw13,tw14}. This determines a f\/iltration of the tangent bundle given by
\begin{gather*}
\mathcal{D} \subset [\mathcal{D},\mathcal{D}] \subset [[\mathcal{D},\mathcal{D}],\mathcal{D}]=TM.
\end{gather*}
The distribution $[\mathcal{D},\mathcal{D}]$ has rank 3 while the full tangent space~$TM$ has rank~5, hence such a~geo\-metry is also known as a~$(2,3,5)$-distribution.
Let $M_{xyzpq}$ denote the 5-dimensional manifold with local coordinates given by $(x,y,z,p,q)$.
The generic 2-plane f\/ield or rank 2 distribution determined by a~function~$F(q)$ of a~single variable with~$F''(q)
\neq 0$ is given by
\begin{gather*}
\mathcal{D}=\operatorname{span}\{\partial_q, \partial_x+p \partial_y +q \partial_p +F(q) \partial_z\}.
\end{gather*}
The fundamental Cartan curvature invariant of this distribution is computed in~\cite{annur} and is found to be the term in the left hand side of~(\ref{6thode}).
It is known that equation~(\ref{6thode})
vanishes when $F(q)=q^m$ and $m \in \big\{{-}1,\frac{1}{3},\frac{2}{3},2\big\}$.
In these cases, the vanishing of the fundamental curvature invariant associated to the distribution $\mathcal{D}$ on $M_{xyzpq}$ implies that the group of local symmetries is the maximal possible given by the split real form of $G_2$. This is the result of~\cite[Corollary~2.1]{annur}.
The authors of \cite{annur} call such generic 2-plane f\/ields with vanishing Cartan curvature invariant symmetric and it is known
that such symmetric or f\/lat distributions are locally equivalent to the f\/lat model $F(q)=q^2$. Nonetheless we are interested in the general solution to~(\ref{6thode}) and it turns out that the ODE can be solved completely and is related to the generalised Chazy equation. To see this, let
$E(q)=F''(q)$ so that the ODE becomes~$4^{\rm th}$ order:
\begin{gather*}
10 E^{(4)}E^3-80E^2E'E'''-51E^2{E''}^2+336 E {E'}^2E''-224{E'}^4=0.
\end{gather*}
Working locally on an open set of $M_{xyzpq}$, we may assume that $E(q)$ is positive on that open set.
Making the substitution $E(q)={\rm e}^{G(q)}$ (if $E(q) <0$, take $E(q)=-{\rm e}^{G(q)}$ instead)
gives
\begin{gather*}
{\rm e}^{4 G(q)}\big(10 G''''-40 G''' G'-21 (G'')^2+54 G'' (G')^2-9 (G')^4\big)=0
\end{gather*}
and taking $G'(q)=j(q)$ gives a $3^{\rm rd}$ order ODE
\begin{gather*}
10 j'''-40 j'' j-21 (j')^2+54 j' j^2-9 j^4=0.
\end{gather*}
Rescaling the ODE by taking $j(q)=\frac{I(q)}{2}$, we can put it into the normal form for the generalised Chazy equation (see \cite{co96})
\begin{gather}\label{3rds}
I'''-2 I'' I+3 (I')^2-\frac{4}{36-\left(\frac{2}{3}\right)^2} \big(6 I' -I^2\big)^2=0
\end{gather}
with \looseness=-1 the Chazy parameter $k^2=\left(\frac{2}{3}\right)^2=\frac{4}{9}$.
The generalised Chazy equation can be solved completely and the solutions give us new families of f\/lat $(2,3,5)$-distributions that are not of the form $F(q)=q^m$. In this article we f\/irst review the solutions to Chazy's equations in Sections~\ref{ce} and~\ref{gce}.
In Section~\ref{dode} we discuss the relationship between~(\ref{6thode}) and a $7^{\rm th}$ order ODE studied by Dunajski and Sokolov in~\cite{ds} and also exhibit a Legendre transform that relates equation~(\ref{3rds}) to another generalised Chazy equation with the Chazy parameter given by $k^2=\left(\frac{3}{2}\right)^2=\frac{9}{4}$.
We compute the solutions to~(\ref{3rds}) in Section~\ref{fos} and present examples of f\/lat $(2,3,5)$-distributions in Section~\ref{examples} using Nurowski's metric. These examples are all explicit.
In~\cite{annur0}, the authors associated to split signature conformal structures on a 4-manifold a circle bundle with the natural structure of a~$(2,3,5)$-distribution. This construction encapsulates the conf\/iguration space of 2 surfaces rolling along one another without slipping and twisting. The authors in~\cite{annur0} then found new examples of f\/lat $(2,3,5)$-distributions that arise from rolling bodies, prompting further search in~\cite{rolling}. The solutions to~(\ref{3rds}) give examples of 4-dimensional split signature metrics that have their An--Nurowski twistor distributions having split $G_2$ as its group of symmetries and we exhibit them in Section~\ref{anexamples}. Let us recall some facts about Chazy's equation and its generalised version.

\section{Chazy's equation}\label{ce}

The study of Chazy's equation is a very rich subject and has received alot of attention because of its connection to other diverse f\/ields such as integrable systems and modular forms. See for instance \cite{ach,chazypara,sdym,co96}.
We will review here some facts about Chazy's equation we need for the paper. Chazy \cite{chazy1,chazy2} studied the nonlinear $3^{\rm rd}$ order ODE
\begin{gather}\label{chazy}
y'''(x)-2 y(x) y''(x)+3 (y'(x))^2=0
\end{gather}
in the context of investigating its Painlev\'e property. Solutions to equation~(\ref{chazy}) turn out to depend on hypergeometric functions. For further details, see~\cite{ach} or~\cite{co96}.
Treat~$x$ as a dependent variable of~$s$ so that
\begin{gather*}
x(s)=\frac{z_2(s)}{z_1(s)},
\end{gather*}
where $z_1(s)$, $z_2(s)$ are linearly independent solutions to the second order hypergeometric dif\/fe\-ren\-tial equation
\begin{gather}\label{hypergeom}
s(1-s)z''+(c-(a+b+1)s)z'-a b z=0.
\end{gather}
Here $a$, $b$, $c$ are constants to be determined.
The general solution to this ODE (\ref{hypergeom}) is given by hypergeometric functions
\begin{gather*}
z(s)=\mu\, {}_2{F}_1(a,b;c;s)+\nu\, {}_2{F}_1(a-c+1,b-c+1;2-c;s)s^{1-c}.
\end{gather*}
Here $\mu$, $\nu$ are constants.
A computation gives
\begin{gather*}
\der x=\frac{z_1 \dot{z}_2-z_2 \dot{z}_1}{(z_1)^2}\der s,
\end{gather*}
where dot denotes derivative with respect to~$s$.
We deduce that
\begin{gather*}
\frac{\der}{\der x}=\frac{(z_1)^2}{z_1 \dot{z}_2-z_2 \dot{z}_1}\frac{\der}{\der s}.
\end{gather*}
Applying the derivative to Chazy's solution for $y$ given by
\begin{gather}\label{original}
y=6 \frac{\der}{\der x}\log z_1=\frac{6 z_1\dot{z}_1}{z_1 \dot{z}_2-z_2 \dot{z}_1},
\end{gather}
we f\/ind that (\ref{chazy})
is satisf\/ied precisely when $(a,b,c)$ is one of
\begin{gather*}
\left(\frac{1}{12}, \frac{1}{12}, \frac{1}{2}\right),\qquad
\left(\frac{1}{12}, \frac{1}{12}, \frac{2}{3}\right),\qquad
\left(\frac{1}{6}, \frac{1}{6}, \frac{2}{3}\right),
\end{gather*}
provided both $a$ and $b$ are non-zero. The equations (\ref{hypergeom}) for the f\/irst two values of $(a,b,c)$ are related by a linear transformation of the form $s \mapsto 1-s$, while the solutions for the second and third values are related by a quadratic transformation (see~\cite[equation~(2)]{hyper}) given by
\begin{gather*}
{}_2{F}_1\left(\frac{1}{6},\frac{1}{6};\frac{2}{3};s\right)={}_2{F}_1\left(\frac{1}{12},\frac{1}{12};\frac{2}{3};4s(1-s)\right).
\end{gather*}
The general solution to (\ref{chazy}) thus depend on hypergeometric functions. If either one of $a$ or $b$ is zero (say $b=0$), solutions to (\ref{chazy}) can be easily and explicitly described. The solutions to (\ref{hypergeom})
with $b=0$ are given by
\begin{gather*}
z(s) =\mu+\nu \, {}_2{F}_{1}(1-c,a+1-c;2-c;s) s^{1-c}
 =\mu-\frac{\nu \pi (c-1)}{\sin(\pi c)}P^{(1-c,~a-c)}_{c-1}(1-2s)s^{1-c},
\end{gather*}
where $P^{(a_1,b_1)}_{n}$ is the Jacobi polynomial.
Taking $z_1(s)=\nu \, {}_2{F}_{1}(1-c,a+1-c;2-c;s) s^{1-c}$ and $z_2(s)=\mu$,
a computation shows that
\begin{gather*}
x(s)=\frac{\mu}{\nu\, {}_2{F}_{1}(1-c,a+1-c;2-c;s)s^{1-c}}
\end{gather*}
and
\begin{gather*}
y(x(s))=-6\frac{\nu}{\mu}\, {}_2{F}_{1}(1-c,a+1-c;2-c;s)s^{1-c}.
\end{gather*}
Switching back to the original independent variable $x$, this gives
\begin{gather*}
y(x)=-\frac{6}{x}
\end{gather*}
as one solution to~(\ref{chazy}). This solution is invariant under translations of the form $x \mapsto x+C$. In~\cite{chazy1,chazy2}, Chazy also observed that
\begin{gather*}
y=-\frac{6}{x+C}-\frac{B}{(x+C)^2}
\end{gather*}
is a solution to~(\ref{chazy}). It is well-known that Chazy's equation and its generalised version can be rewritten as a f\/irst order system. This provides dif\/ferent parametrisations of~$y$, in addition to the solution~(\ref{original}) originally given by Chazy. This will be discussed in Section~\ref{fos}. The method discussed here can also be applied to the generalised Chazy equation.

\section{Generalised Chazy equations}\label{gce}

The generalised Chazy equation is given by
\begin{gather}\label{gc0}
y'''(x)-2 y''(x) y(x)+3 (y'(x))^2-\frac{4}{36-k^2} (6 y'(x) -y(x)^2)^2=0
\end{gather}
 for $k \neq \pm 6$. We have the following:

\begin{Proposition} \label{gct}
Let
$x(s)=\frac{z_2(s)}{z_1(s)}$ where $z_1(s)$, $z_2(s)$ are linearly independent solutions to the hypergeometric differential equation~\eqref{hypergeom}
where
$(a,b,c)$ is one of
\begin{gather*}
\left(\frac{k-6}{12k }, \frac{k+6}{12k}, \frac{1}{2}\right),\qquad
\left(\frac{k-6}{6k}, \frac{k+6}{6k}, \frac{2}{3}\right),\qquad
\left(\frac{k-6}{12k }, \frac{k+6}{12k},\frac{2}{3}\right).
\end{gather*}
Then
\begin{gather*}
y(x(s))=6 \frac{\der}{\der x}\log z_1= \frac{6z_1\dot{z}_1}{z_1 \dot{z}_2-z_2 \dot{z}_1}
\end{gather*}
satisfies equation \eqref{gc0}.
\end{Proposition}

\begin{proof}
Analogous to solving Chazy's equation (\ref{chazy}), we f\/ind that the generalised equation (\ref{gc0}) holds provided
\begin{gather*}
6 a b (z_1)^8\Big( 6((a-b)k-6(a+b))((a-b)k+6(a+b))s^2\\
\quad{}+\big((24ab-12 (a+b)c+ 5(a+b)+(2c-1))k^2+432(a+b)c-180(a+b)-72c+36\big)s\\
\qquad{} +(k-6)(k+6)(2c-1)(3c-2)\Big)=0.
\end{gather*}
For $a, b\neq 0$, solving the system of equations
\begin{gather*}
6((a-b)k-6(a+b))((a-b)k+6(a+b)) =0,\\
(24ab-12 (a+b)c+5(a+b)+(2c-1))k^2 +432(a+b)c-180(a+b)-72c+36=0,\\
(k-6)(k+6)(2c-1)(3c-2)=0,
\end{gather*}
gives the list of $(a,b,c)$ as above. We exclude the case where $(a,b,c)=\big(0,0,\frac{1}{2}\big)$. Note that interchanging~$a$ and~$b$ gives the same solution so that the full list is symmetric in~$a$ and~$b$.
\end{proof}

\looseness=-1
When either $a$ or $b$ is zero, we again get $y(x)=-\frac{6}{x}$ as a solution. In~\cite{chazy1,chazy2}, Chazy noted that
\begin{gather}\label{nsol}
y=\frac{k-6}{2(x+C)}-\frac{k+6}{2(x+B)}
\end{gather}
is also a solution to (\ref{gc0}).
As a~corollary to Proposition~\ref{gct}, we have

\begin{Corollary} \label{I235}
Let $q(s)=\frac{z_2(s)}{z_1(s)}$ where $z_1(s)$, $z_2(s)$ are linearly independent solutions to the hypergeometric differential equation~\eqref{hypergeom} with $(a,b,c)$ one of
\begin{gather*}
\left(-\frac{2}{3 }, \frac{5}{6}, \frac{1}{2}\right),\qquad
\left(-\frac{4}{3}, \frac{5}{3}, \frac{2}{3}\right),\qquad
\left(-\frac{2}{3}, \frac{5}{6 }, \frac{2}{3}\right).
\end{gather*}
Then
\begin{gather*}
I(q(s))=6 \frac{\der}{\der q}\log z_1=\frac{6z_1\dot{z}_1}{z_1 \dot{z}_2-z_2 \dot{z}_1}
\end{gather*}
satisfies equation \eqref{3rds}.
\end{Corollary}

A Painlev\'e type analysis of equation (\ref{gc0}) as done in \cite{co96} shows that the leading orders for analytic solutions to (\ref{gc0}) occur at $-6$, $-3+\frac{k}{2}$ or $-3-\frac{k}{2}$. This corresponds to solutions of~(\ref{gc0}) given by
\begin{gather*}
y(x)=-\frac{6}{x}, \qquad y(x)=\frac{-3+\frac{k}{2}}{x}, \qquad y(x)=\frac{-3-\frac{k}{2}}{x}.
\end{gather*}
These solutions are invariant under translations of the form $x \mapsto x+C$.
In the case of $k=\pm \frac{2}{3}$ obtained in~(\ref{3rds}), we have
\begin{gather*}
I(q)=-\frac{6}{q}, \qquad I(q)=-\frac{10}{3q}, \qquad I(q)=-\frac{8}{3 q}.
\end{gather*}
Along with the zero solution $I(q)=0$,
these solutions correspond respectively (modulo constants of integration) to the well-known explicit solutions to~(\ref{6thode}):
\begin{gather*}
F(q)=q^{-1}, \qquad F(q)=q^{\frac{1}{3}}, \qquad F(q)=q^{\frac{2}{3}}, \qquad F(q)=q^2.
\end{gather*}
For these functions of a single variable $q$ the associated $(2,3,5)$-distributions have vanishing Cartan invariant and therefore have~$G_2$ as their local symmetry.

\section{Relationship to ODE studied by Dunajski and Sokolov}\label{dode}

For the function $y=y(t)$, the $7^{\rm th}$ order nonlinear ODE studied in \cite{ds} is given by
\begin{gather}
10 \big(y^{(3)}\big)^3y^{(7)}-70\big(y^{(3)}\big)^2y^{(4)} y^{(6)}-49\big(y^{(3)}\big)^2\big(y^{(5)}\big)^2\nonumber\\
\qquad{} +280 \big(y^{(3)}\big)\big(y^{(4)}\big)^2y^{(5)}-175\big(y^{(4)}\big)^4=0.\label{7thode}
\end{gather}
This is the unique $7^{\rm th}$ order ODE admitting the submaximal contact symmetry group of dimension ten (see~\cite{ds,olver}) and its relationship to equation~(\ref{6thode}) was originally explored in~\cite{annur}. It is instructive to consider the $6^{\rm th}$ order ODE (for the Legendre transformation later on):
\begin{gather}
 10 \big(H^{(2)}\big)^3H^{(6)}-70\big(H^{(2)}\big)^2H^{(3)} H^{(5)}-49\big(H^{(2)}\big)^2\big(H^{(4)}\big)^2\nonumber\\
\qquad{} +280 \big(H^{(2)}\big)\big(H^{(3)}\big)^2H^{(4)}-175\big(H^{(3)}\big)^4=0\label{ds6}
\end{gather}
with $H(t)=y'(t)$. Let us show that this ODE can be reduced to a generalised Chazy equation.
Again working locally in an open set where $y'''(t)$ is non-zero, and assuming $y'''(t)$ to be positive, we can make the substitution ${\rm e}^{p(t)}=y^{(3)}$ to get
\begin{align}\label{4thode}
{\rm e}^{p(t)} \big(10p^{(4)}-30p'p'''-19(p'')^2+32 (p')^2p''-4(p')^4\big)=0.
\end{align}
We note that this $4^{\rm th}$ order ODE historically appears in~\cite[Section~XII, formula~(12)]{cartan1910}, where it f\/irst arises as the obstruction to integrability for $(2,3,5)$-distributions of the form~$\mathcal{D}_{F(q)}$. This will be made clear below once we show that~(\ref{ds6}) is the Legendre transform of~(\ref{6thode})~\cite{annur} and we will discuss this further in Section~\ref{examples}.
Thus, for $v(t)=p'(t)$, we obtain the third order ODE
\begin{align}\label{ds3ode}
10v'''-30 vv''-19 (v')^2+32 v' v^2-4v^4=0.
\end{align}
Rescaling $v(t)$ by $u(t)=\frac{3}{2}v(t)$, we put~(\ref{ds3ode}) into the normal form
\begin{gather}\label{gc2}
u'''-2 u'' u+3 (u')^2-\frac{4}{36-\left(\frac{3}{2}\right)^2} \big(6 u' -u^2\big)^2=0.
\end{gather}
We therefore see that the ODE that Dunajski and Sokolov study in \cite{ds} reduces to a generalised Chazy equation~(\ref{gc2}) with parameter $k'=\pm\frac{3}{2}$, related to the generalised Chazy equation~(\ref{3rds}) just by taking the reciprocals $(k')^2=\frac{1}{k^2}$ of the corresponding parameters.

Let $t(s)=\frac{w_2(s)}{w_1(s)}$ where $w_1(s)$, $w_2(s)$ are linearly independent solutions to the hypergeometric dif\/ferential equation~(\ref{hypergeom}) with $(a,b,c)$ one of
\begin{gather*}
\left(-\frac{1}{4 }, \frac{5}{12}, \frac{1}{2}\right),\qquad
\left(-\frac{1}{4}, \frac{5}{12}, \frac{2}{3}\right),\qquad
\left(-\frac{1}{2 }, \frac{5}{6},\frac{2}{3}\right).
\end{gather*}
The solution to (\ref{gc2}) is then given by $u=6\frac{\der}{\der t}\log w_1$. A similar leading order analysis as before shows that the leading orders occur at
\begin{gather*}
-6, \quad -\frac{9}{8},\quad -\frac{15}{8}.
\end{gather*}
This corresponds to solutions of (\ref{gc2}) given by
\begin{gather*}
u(t)=-\frac{6}{t}, \qquad u(t)=-\frac{9}{4t}, \qquad u(t)=-\frac{15}{4t}.
\end{gather*}
Along with the zero solution $u(t)=0$,
these correspond respectively (modulo constants of integration) to solutions of~(\ref{ds6}) given by
\begin{gather*}
H(t)=t^{-2}, \qquad H(t)=t^{\frac{1}{2}}, \qquad H(t)=t^{-\frac{1}{2}}, \qquad H(t)=t^2.
\end{gather*}

In~\cite[Proposition~2.2]{annur}, it is shown that a Legendre transformation takes (\ref{6thode}) to (\ref{7thode}). Hence we may hypothesise that amongst all $3^{\rm rd}$ order generalised Chazy equations, only those with the parameters $k'=\pm \frac{3}{2}$, $k=\pm \frac{2}{3}$ have in addition solutions that can be obtained from the dual equation via a Legendre transform. 
\begin{Proposition}[\protect{\cite[Proposition~2.2]{annur}}]\label{leg1}
Consider the Legendre transformation
\begin{gather*}
F(q)+H(t)=q t.
\end{gather*}
Then $F(q)$ satisfies the ODE \eqref{6thode} iff $H(t)$ satisfies the ODE~\eqref{ds6}.
\end{Proposition}

\begin{proof}
Applying the exterior derivative to the relation gives
\begin{gather*}
(F'-t) \der q+(H'-q) \der t=0,
\end{gather*}
so that we take $F'=t$, $H'=q$ and applying
\begin{gather*}
 \frac{\der }{\der q} = \frac{1}{H''}\frac{\der}{\der t}
\end{gather*}
we obtain $F''=\frac{1}{H''}$, $F^{(3)}=-\frac{H^{(3)}}{(H'')^3}$, etc.
A computation shows that the $6^{\rm th}$ order ODE~(\ref{6thode}) holds for $F$ if\/f (\ref{ds6}) holds for $H$.
\end{proof}

In light of the solutions obtained by solving the generalised Chazy equations, we can pass to
\begin{gather}\label{intp}
F(q)=\iint {\rm e}^{\int \frac{I(q)}{2} \der q}\der q \der q,
\end{gather}
where $q=\frac{z_2(s)}{z_1(s)}$ and $I(q)=6 \frac{\der}{\der q}\log z_1$ are given in Corollary~\ref{I235}. This gives
\begin{gather*}
F(q)=\iint (z_1)^3 \der q \der q.
\end{gather*}
Similarly, for the dual equation (\ref{ds6}) under the Legendre transform we pass to
\begin{gather*}
H(t)=\iint {\rm e}^{\int \frac{2u(t)}{3} \der t}\der t \der t,
\end{gather*}
where $t=\frac{w_2(s)}{w_1(s)}$ and $u(t)=6 \frac{\der}{\der t}\log w_1$ are solutions to (\ref{gc2}). This gives
\begin{gather*}
H(t)=\iint (w_1)^4 \der t \der t.
\end{gather*}
We have
\begin{Lemma}\label{legendre}\sloppy
There exists a Legendre transformation between Chazy's solutions of~\eqref{3rds} and~\eqref{gc2} given by taking
\begin{gather*}
w_1(s)=z_1^{-\frac{3}{4}}, \qquad w_2(s)=(z_1)^{-\frac{3}{4}} \int (z_1) (\dot{z}_2 z_1-\dot{z}_1z_2) \der s.
\end{gather*}
This defines a mapping
\begin{gather*}
q=\frac{z_2(s)}{z_1(s)} \mapsto t=\frac{w_2(s)}{w_1(s)}=\int z_1 (\dot{z}_2 z_1-\dot{z}_1z_2) \der s.
\end{gather*}
If $I(q)=6 \frac{\der}{\der q}\log z_1$ solves~\eqref{3rds} where $z_1(s)$ and $z_2(s)$ are given in Corollary~{\rm \ref{I235}}, then $u(t)=6 \frac{\der}{\der t}\log w_1$ solves the dual ODE~\eqref{gc2}.
Consequently, if
$F(q)=\iint (z_1)^3 \der q \der q$ solves~\eqref{6thode}, then
\begin{gather*}
H(t)=\iint (w_1)^4 \der t\der t=\iint (z_1)^{-2} (\dot{z}_2 z_1-\dot{z}_1z_2) \der s
 \, z_1(\dot{z}_2 z_1-\dot{z}_1z_2) \der s
\end{gather*}
solves the $6^{\text{th}}$ order ODE~\eqref{ds6}.
For the converse, the Legendre transform is given by
\begin{gather*}
z_1(s)=w_1^{-\frac{4}{3}}, \qquad z_2(s)=(w_1)^{-\frac{4}{3}} \int (w_1)^2 (\dot{w}_2 w_1-\dot{w}_1w_2) \der s.
\end{gather*}
This sends
\begin{gather*}
t=\frac{w_2(s)}{w_1(s)} \mapsto q=\frac{z_2(s)}{z_1(s)}=\int (w_1)^2 (\dot{w}_2 w_1-\dot{w}_1w_2) \der s.
\end{gather*}
In particular, if $u(t)=6 \frac{\der}{\der t}\log w_1$ solves the dual ODE~\eqref{gc2}, then $I(q)=6 \frac{\der}{\der q}\log z_1$ solves~\eqref{3rds}.
Hence, if
$H(t)=\iint (w_1)^4 \der t \der t$ solves the $6^{\text{th}}$ order ODE~\eqref{ds6}, then
\begin{gather*}
F(q)=\iint (z_1)^3 \der q\der q=\iint (w_1)^{-2} (\dot{w}_2 w_1-\dot{w}_1w_2) \der s
\, (w_1)^2(\dot{w}_2 w_1-\dot{w}_1w_2) \der s
\end{gather*}
solves~\eqref{6thode}.
\end{Lemma}

\begin{proof}
We observe that as a consequence of Chazy's solutions, the Legendre transform in Proposition \ref{leg1} gives
\begin{gather*}
\frac{w_2}{w_1}=t= F'=\int (z_1)^3 \der q=\int z_1 (\dot{z}_2 z_1-\dot{z}_1 z_2)\der s,\\
\frac{z_2}{z_1}=q= H'=\int (w_1)^4 \der t=\int (w_1)^2 (\dot{w}_2 w_1-\dot{w}_1 w_2)\der s,
\end{gather*}
and therefore
\begin{gather*}
\frac{\dot{w}_2w_1-\dot{w}_1 w_2}{(w_1)^2} =z_1 (\dot{z}_2 z_1-\dot{z}_1 z_2),\\
\frac{\dot{z}_2z_1-\dot{z}_1 z_2}{(z_1)^2} =(w_1)^2 (\dot{w}_2 w_1-\dot{w}_1 w_2).
\end{gather*}
Together this yields $(z_1)^3=(w_1)^{-4}$,
from which we deduce
\begin{gather*}
w_1=z_1^{-\frac{3}{4}}\qquad\mbox{and}\qquad w_2=w_1\int z_1 (\dot{z}_2 z_1-\dot{z}_1 z_2)\der s=z_1^{-\frac{3}{4}}\int z_1 (\dot{z}_2 z_1-\dot{z}_1 z_2)\der s.
\end{gather*}
For the converse, we f\/ind
\begin{gather*}
z_1=w_1^{-\frac{4}{3}}\qquad\mbox{and}\qquad
z_2=z_1\int (w_1)^2 (\dot{w}_2 w_1-\dot{w}_1 w_2)\der s=w_1^{-\frac{4}{3}}\int (w_1)^2 (\dot{w}_2 w_1-\dot{w}_1 w_2)\der s.
\end{gather*}
The rest follows from a routine computation.
\end{proof}

In \cite[formula~(8)]{ds}, a family of solutions to~(\ref{7thode})
is found to be given by the algebraic curve
\begin{gather*}
(y+f(t))^2=(t-a)(t-b)^3,
\end{gather*}
with $a \neq b$, and $f(t)$ a quadratic.
This gives
\begin{gather*}
y=\pm \sqrt{(t-a)(t-b)^3}-f(t).
\end{gather*}
We obtain a solution to (\ref{7thode}) with
\begin{gather*}
y^{(3)}=\pm \frac{3}{8}\frac{(t-b)^6(a-b)^3}{(y+f)^5}=\pm \frac{3(a-b)^3}{8(t-a)^2(y+f)}.
\end{gather*}
We f\/ind that for this solution, it yields
\begin{gather*}
u(t)=\frac{3}{4}\frac{5b+3a-8t}{(t-a)(t-b)}
=-\frac{15}{4(t-a)}-\frac{9}{4(t-b)}
\end{gather*}
as a solution to the generalized Chazy's equation with parameter $k^2=\frac{9}{4}$. This corresponds to the solution given by Chazy in~(\ref{nsol}). It will be interesting to determine the solutions of~(\ref{gc2}) from the general solution given by~\cite[formula~(13)]{ds}.

\section[First order system and dif\/ferent parametrisations of Chazy's equations]{First order system and dif\/ferent parametrisations\\ of Chazy's equations}\label{fos}

In this section, we f\/irst show that the generalised Chazy equation is equivalent to solving a third order dif\/ferential equation involving the Schwarzian derivative and a potential term~$V(s)$.
It is well-known that solutions to the generalised Chazy equation~(\ref{gc0}) can be rewritten as a f\/irst order system. For further details, see~\cite{ach}. The f\/irst order system provides dif\/ferent parametrisations of the solutions, in addition to the one given by~(\ref{original}). We compute the solutions to the generalised Chazy equation~(\ref{3rds}) with $k=\pm\frac{2}{3}$ for the dif\/ferent parametrisations below and present them in Tables~\ref{table1}, \ref{table2} and~\ref{table3}. We also show how these solutions are related to one another by algebraic transformation of hypergeometric functions.

Let $\Omega_1$, $\Omega_2$, $\Omega_3$ be functions of~$q$.
Let dot denote dif\/ferentiation with respect to~$q$.
Then consider
\begin{gather}
\dot{\Omega}_1=\Omega_2 \Omega_3-\Omega_1(\Omega_2+\Omega_3)+\tau^2,\nonumber\\
\dot{\Omega}_2=\Omega_3 \Omega_1-\Omega_2(\Omega_3+\Omega_1)+\tau^2,\nonumber\\
\dot{\Omega}_3=\Omega_1 \Omega_2-\Omega_3(\Omega_1+\Omega_2)+\tau^2,\label{chazy1}
\end{gather}
where
\begin{gather*}
\tau^2=\alpha^2 (\Omega_1-\Omega_2)(\Omega_3-\Omega_1)+\beta^2 (\Omega_2-\Omega_3)(\Omega_1-\Omega_2)+\gamma^2(\Omega_3-\Omega_1)(\Omega_2-\Omega_3)
\end{gather*}
and $\alpha$, $\beta$, $\gamma$ are constants. Introducing the parameter
\begin{gather*}
s(q)=\frac{\Omega_1-\Omega_3}{\Omega_2-\Omega_3},
\end{gather*}
we f\/ind that
\begin{gather*}
\Omega_1 =-\frac{1}{2}\frac{\der}{\der q}\log\frac{\dot{s}}{s(s-1)},\qquad
\Omega_2 =-\frac{1}{2}\frac{\der}{\der q}\log\frac{\dot{s}}{s-1},\qquad
\Omega_3 =-\frac{1}{2}\frac{\der}{\der q}\log\frac{\dot{s}}{s}.
\end{gather*}
The system of equations (\ref{chazy1}) are satisf\/ied if\/f
\begin{gather}\label{schwarzian}
\{s,q\}+\frac{\dot{s}^2}{2}V(s)=0,
\end{gather}
where
\begin{align*}
\{s,q\}=\frac{\der}{\der q}\left(\frac{\ddot{s}}{\dot{s}}\right)-\frac{1}{2}\left(\frac{\ddot{s}}{\dot{s}}\right)^2
\end{align*}
is the Schwarzian derivative of $s(q)$
and
\begin{gather*}
V(s)=\frac{1-\beta^2}{s^2}+\frac{1-\gamma^2}{(s-1)^2}+\frac{\beta^2+\gamma^2-\alpha^2-1}{s(s-1)}.
\end{gather*}
Switching independent and dependent variables in (\ref{schwarzian}), we have
\begin{gather*}
\{s,q\}=-(\dot{s})^2\{q,s\},
\end{gather*}
so that the dual of (\ref{schwarzian}) is
\begin{gather*}
\{q,s\}-\frac{1}{2}V(s)=0.
\end{gather*}
The general solution is given by
\begin{gather*}
q(s)=\frac{u_2(s)}{u_1(s)}
\end{gather*}
where $u_1$, $u_2$ are linearly independent solutions of the second order ODE
\begin{gather}\label{ode1}
u''+\frac{1}{4}V(s) u=0.
\end{gather}
The general solution of (\ref{ode1}) suggests taking
\begin{gather*}
u(s)=(s-1)^{\frac{1-\gamma}{2}}s^{\frac{1-\beta}{2}}z(s)
\end{gather*}
(cf.\ \cite{ach2003}) to transform (\ref{ode1}) into the hypergeometric dif\/ferential equation (\ref{hypergeom}) with
\begin{gather*}
a =\frac{1}{2}(1-\alpha-\beta-\gamma),\qquad
b =\frac{1}{2}(1+\alpha-\beta-\gamma),\qquad
c =1-\beta.
\end{gather*}
In \cite{ach}, it was determined that taking
\begin{gather}\label{yes}
y=-2(\Omega_1+\Omega_2+\Omega_3)=\frac{\der}{\der q}\log\frac{\dot{s}^3}{s^2(s-1)^2}
\end{gather}
gives solutions to the generalised Chazy equation
(\ref{gc0}) with variable $q$ whenever $\alpha=\beta=\gamma=\frac{2}{k}$ or $\alpha=\frac{2}{k}$, $\beta=\gamma=\frac{1}{3}$ and its cyclic permutations. For $\alpha=\beta=\gamma=\frac{2}{k}$ this gives $(a,b,c)=\big(\frac{k-6}{2k},\frac{k-2}{2k}, \frac{k-2}{k}\big)$. For $\alpha=\frac{2}{k}$, $\beta=\gamma=\frac{1}{3}$ with cyclic permutations, this gives respectively
\begin{gather*}
(a,b,c)=\left(\frac{k-6}{6k}, \frac{k+6}{6k}, \frac{2}{3}\right), \qquad
\left(\frac{k-6}{6k}, \frac{k-2}{2k}, \frac{k-2}{k}\right), \qquad
\left(\frac{k-6}{6k}, \frac{k-2}{2k}, \frac{2}{3}\right),
\end{gather*}
with only the f\/irst coinciding with the list in Proposition~\ref{gct}. This suggests that the solutions to~(\ref{schwarzian}) are more general than the solution of the form $y=6\frac{\der}{\der q}\log z_1$ given by Chazy~\cite{chazy2}. We can express Chazy's solution in terms of $s(q)$ as follows.
A computation of the Wronskian of linearly independent solutions $z_1$, $z_2$ to~(\ref{hypergeom}) gives
\begin{gather*}
W(z_1,z_2)=z_1 \dot{z}_2-z_2 \dot{z}_1=w_0 (s-1)^{c-a-b-1} s^{-c}
\end{gather*}
for some non-zero constant $w_0$. The latter equality holds by solving the f\/irst order dif\/ferential equation
\begin{gather*}
W'=-\frac{c-(a+b+1)s}{s(1-s)}W.
\end{gather*}
See for instance \cite{comp}. From $q(s)=\frac{z_2(s)}{z_1(s)}$, we f\/ind that $\frac{\der}{\der q}=\frac{(z_1)^2}{z_1\dot{z}_2-z_2\dot{z}_1}\frac{\der}{\der s}$.
Applying this derivative to~$s(q)$, we obtain
\begin{gather*}
s'(q)=\frac{\der}{\der q}s(q)=\frac{(z_1)^2}{z_1\dot{z}_2-z_2\dot{z}_1}=\frac{(z_1)^2}{W(z_1,z_2)}.
\end{gather*}
Hence
\begin{align*}
s''(q)=\frac{\der}{\der q}s'(q)&=s'(q)\frac{\der}{\der s}\left(\frac{(z_1)^2}{W(z_1,z_2)}\right)
 =s'(q)\left(\frac{2 z_1 \dot{z}_1}{W(z_1,z_2)}-\frac{(z_1)^2 W'(z_1,z_2)}{W(z_1,z_2)^2}\right)\\
&=s'(q)\left(2 s'(q) \frac{\dot{z}_1}{z_1}-s'(q)\frac{ W'(z_1,z_2)}{W(z_1,z_2)}\right)\\
&=s'(q)\left(2 s'(q) \frac{\dot{z}_1}{z_1}+s'(q)\frac{ c-(a+b+1)s(q)}{s(q)(1-s(q))}\right)\\
&=s'(q)\left(2 s'(q) \frac{\dot{z}_1}{z_1}+s'(q)\frac{ c(1-s(q))-(a+b+1-c)s(q)}{s(q)(1-s(q))}\right),
\end{align*}
and we get
\begin{gather*}
\frac{\ddot{s}}{\dot{s}}
=2 \dot{s} \frac{\dot{z}_1}{z_1}+c\frac{\dot{s}}{s}-(a+b+1-c)\frac{\dot{s}}{1-s}.
\end{gather*}
Therefore we have
\begin{align*}
y=6\frac{\der}{\der q}\log z_1&=6\frac{(z_1)^2}{W(z_1,z_2)}\frac{\dot{z}_1}{z_1}=6 \dot{s}\frac{\dot{z}_1}{z_1}
 =3\frac{\ddot{s}}{\dot{s}}-3c\frac{\dot{s}}{s}+3(a+b+1-c)\frac{\dot{s}}{1-s}\\
&=3\frac{\der}{\der q}\log{\dot{s}}-3c\frac{\der}{\der q}\log s-3(a+b+1-c)\frac{\der}{\der q}\log (1-s)\\
&=3\frac{\der}{\der q}\log \frac{\dot{s}}{s^c(1-s)^{a+b+1-c}}
=\frac{1}{2}\frac{\der}{\der q}\log \frac{\dot{s}^6}{s^{6c}(1-s)^{6(a+b+1-c)}}.
\end{align*}
A comparison of Chazy's formula for $y=6\frac{\der}{\der q}\log z_1$ with the formula for~$y$ in~(\ref{yes})
suggests taking $c=\frac{2}{3}$, $a+b=\frac{1}{3}$. This is satisf\/ied by $(a,b,c)=\big(\frac{k-6}{6k}, \frac{k+6}{6k}, \frac{2}{3}\big)$ in Proposition~\ref{gct}.

For $k=\frac{2}{3}$ as in (\ref{3rds}), we get the following solutions for $u$ given in Table~\ref{table1}.
 \begin{table}[h]\centering
 \caption{}\label{table1}\vspace{1mm}

 \begin{tabular}{|c|c|c|} \hline
 $(\alpha,\beta,\gamma)$ & $(a,b,c)$ & General solution to $u''+\frac{1}{4}V(s) u=0$ \tsep{2pt}\bsep{2pt}\\ \hline\hline
$(3,3,3)$ & $(-4,-1,-2)$& $c_1 \frac{2s-1}{s(s-1)}+c_2 \frac{s^2(s-2)}{s-1}$ \tsep{4pt}\bsep{2pt} \\ \hline
$\left(3, \frac{1}{3}, \frac{1}{3}\right)$ & $\left(-\frac{4}{3}, \frac{5}{3}, \frac{2}{3}\right)$ & $c_1 (3s-2)s^{\frac{2}{3}} (s-1)^{\frac{1}{3}}+c_2 (3s-1) s^{\frac{1}{3}}(s-1)^{\frac{2}{3}}$ \tsep{4pt}\bsep{2pt}\\ \hline
$\left(\frac{1}{3 }, 3,\frac{1}{3}\right)$ & $\left(-\frac{4}{3},-1,-2\right)$ &$c_1 \frac{2s-3}{s} (s-1)^{\frac{1}{3}}+c_2 \frac{s-3}{s} (s-1)^{\frac{2}{3}}$ \tsep{4pt}\bsep{2pt}\\ \hline
$\left(\frac{1}{3}, \frac{1}{3},3\right)$ & $\left(-\frac{4}{3},-1,\frac{2}{3}\right)$ &$c_1 \frac{2s+1}{s-1} s^{\frac{1}{3}}+c_2 \frac{s+2}{s-1} s^{\frac{2}{3}}$
\tsep{4pt}\bsep{2pt}\\ \hline
 \end{tabular}
 \end{table}
For this Chazy parameter, the hypergometric series truncate and the solutions to~(\ref{ode1}) can be given by elementary functions. We have
\begin{gather*}
{}_2{F}_1(-4,-1;-2;s) =1-2s,\qquad
{}_2{F}_1\left(-\frac{4}{3},\frac{5}{3};\frac{2}{3};s\right) =\frac{3s^2-4s+1}{(1-s)^{\frac{2}{3}}},\\
{}_2{F}_1\left(-\frac{4}{3},-1;-2;s\right) =1-\frac{2}{3}s,\qquad
{}_2{F}_1\left(-\frac{4}{3},-1;\frac{2}{3};s\right) =1+2s.
\end{gather*}
Moreover, the solutions in each row are related to one another by algebraic transformations of hypergeometric functions. The solutions in the second, third and fourth rows of Table~\ref{table1} can be obtained from the f\/irst by a cubic transformation of hypergeometric functions (see~\cite[formula~(23)]{hyper}).
Explicitly, to show how the solution in the third row is related to the f\/irst, let~$\omega$ be a~cube root of unity (solution to $\omega^2+\omega+1=0$) and consider the map
\begin{gather*}
s \mapsto t=\frac{3 (2\omega+1) s(s-1)}{(s+\omega)^3}
\end{gather*}
and the relation
\begin{gather*}
\tilde{z}(t)=(1+\omega s)^{-4}z(s).
\end{gather*}
Then $\tilde{z}(t)$ is a solution to the hypergeometric dif\/ferential equation~(\ref{hypergeom}) with $(a,b,c)=\big({-}\frac{4}{3},-1$, $-2\big)$ if\/f $z(s)$ solves~(\ref{hypergeom}) with $(a,b,c)=(-4,-1,-2)$.
The solutions in the last three rows are related to one another by fractional linear transformations. The symmetry $s \mapsto 1-s$ interchanges~$\beta$ and $\gamma$ in $V(s)$ and hence transforms the solution in the~$3^{\rm rd}$ row to the solution given in the $4^{\rm th}$ row while $s \mapsto t=\frac{s}{s-1}$ and the relation $\tilde{z}(t)=(1-s)^{-\frac{4}{3}} z(s)$ transforms the solution in the $4^{\rm th}$ row to those in the $2^{\rm nd}$.

A dif\/ferent parametrisation of Chazy's equations (cf.~\cite{chazypara}) is also given by
\begin{align}\label{schwarz}
y=-\Omega_1-2\Omega_2-3\Omega_3=\frac{1}{2}\frac{\der}{\der q}\log \frac{\dot{s}^6}{s^4(s-1)^{3}}.
\end{align}
A comparison with Chazy's formula (\ref{original})
yields $(a,b,c)=\big(\frac{k-6}{12k}, \frac{k+6}{12k}, \frac{2}{3}\big)$ from Proposition \ref{gct}. The solution of the form~(\ref{schwarz}) solves the generalised Chazy equation (\ref{gct}) whenever $(\alpha, \beta, \gamma)$ in~(\ref{schwarzian}) is given by
\begin{gather*}
\left(\frac{1}{k}, \frac{1}{3}, \frac{1}{2}\right),\qquad
\left(\frac{1}{k}, \frac{2}{k}, \frac{1}{2}\right)\qquad \text{or}\qquad
 \left(\frac{1}{k}, \frac{1}{3},\frac{3}{k}\right).
\end{gather*}
The solution to (\ref{schwarzian}) with $(\alpha, \beta,\gamma)=\big(\frac{1}{k}, \frac{1}{3}, \frac{1}{2}\big)$ is given by the Schwarz function~$J$
(see~\cite{ach,ach2003}).

Considering the symmetry $s=1-K$ brings (\ref{schwarz}) to
\begin{gather}\label{schwarz2}
y=-\Omega_1-3\Omega_2-2\Omega_3=\frac{1}{2}\frac{\der}{\der q}\log \frac{-\dot{K}^6}{(1-K)^4K^{3}}
\end{gather}
and comparing with Chazy's formula gives $(a,b,c)=\big(\frac{k-6}{12k}, \frac{k+6}{12k}, \frac{1}{2}\big)$
from Proposition~\ref{gct}.
The solution of the form~(\ref{schwarz2}) solves~(\ref{gct}) whenever $(\alpha, \beta, \gamma)$ in~(\ref{schwarzian}) is given by
\begin{gather*}
\left(\frac{1}{k}, \frac{1}{2}, \frac{1}{3}\right),\qquad
\left(\frac{1}{k}, \frac{1}{2}, \frac{2}{k}\right) \qquad \text{or} \qquad
 \left(\frac{1}{k}, \frac{3}{k},\frac{1}{3}\right).
\end{gather*}
The symmetry $s=1-K$ permutes $\beta$ and $\gamma$ in the formula for~$V(s)$ in~(\ref{schwarzian}).
For~$k=\frac{2}{3}$, the solutions to~(\ref{ode1}) with the parametrisation by~(\ref{schwarz}) are presented in Table~\ref{table2}.

 \begin{table}[h]\centering
 \caption{} \label{table2}\vspace{1mm}
 \begin{tabular}{|c|c|c|} \hline
 $(\alpha,\beta,\gamma)$&$(a,b,c)$ & General solution to $u''+\frac{1}{4}V(s) u=0$ \tsep{2pt}\bsep{2pt}\\ \hline\hline
$\left(\frac{3}{2}, \frac{1}{3}, \frac{1}{2}\right)$ & $\left(\frac{5}{6}, -\frac{2}{3}, \frac{2}{3}\right)$ & $c_1 (s-1)^{\frac{1}{4}}s^{\frac{1}{2}}P^{\frac{1}{3}}_{1}(\sqrt{1-s})+c_2 (s-1)^{\frac{1}{4}}s^{\frac{1}{2}}Q^{\frac{1}{3}}_{1}(\sqrt{1-s})$
\tsep{6pt}\bsep{2pt}\\ \hline
$\left(\frac{3}{2 }, 3,\frac{1}{2}\right)$ & $\left(-\frac{1}{2}, -2, -2\right)$ &$c_1 \frac{(s-1)^{\frac{3}{4}}}{s}+c_2 \frac{(s-1)^{\frac{1}{4}}}{s} (s^2+4s-8)$ \tsep{6pt}\bsep{2pt}\\ \hline
$\left(\frac{3}{2}, \frac{1}{3},\frac{9}{2}\right)$ & $\left(-\frac{7}{6}, -\frac{8}{3}, \frac{2}{3}\right)$& $c_1 s^{\frac{2}{3}}(s-1)^{\frac{11}{4}}\,{}_2{F}_1\big(\frac{13}{6},\frac{11}{3};\frac{4}{3};s\big)+c_2 s^{\frac{1}{3}}(s-1)^{\frac{11}{4}}\, {}_2{F}_1(\frac{11}{6},\frac{10}{3};\frac{2}{3};s)$ \tsep{4pt}\bsep{2pt}\\ \hline
 \end{tabular}

 \end{table}
We also note that we have
\begin{gather*}
{}_2{F}_1\left(\frac{5}{6},-\frac{2}{3};\frac{2}{3};s\right) =\frac{\Gamma\left(\frac{2}{3}\right)^3 \sqrt{3}}{\pi}P^{\left(-\frac{1}{3},-\frac{1}{2}\right)}_{\frac{2}{3}}(1-2s),\\
{}_2{F}_1\left(-\frac{1}{2},-2;-2;s\right) =\frac{1}{8}\big(8-4s-s^2\big),\\
{}_2{F}_1\left(-\frac{7}{6},-\frac{8}{3};\frac{2}{3};s\right) =\frac{10}{7}\frac{\Gamma\left(\frac{2}{3}\right)^3 \sqrt{3}}{\pi}P^{\left(-\frac{1}{3},-\frac{9}{2}\right)}_{\frac{8}{3}}(1-2s).
\end{gather*}
The algebraic transformations relating the solutions between the rows are given as follows.
The map that takes the solution from the second row to the solution in the f\/irst row of Table~\ref{table2} is a~composition of fractional linear transformations and cubic transformation due to Goursat (see~\cite[formulas~(20)--(22)]{hyper} and~\cite[formula~(123)]{goursat}).
To see this, we f\/irst consider the map
\begin{gather*}
s \mapsto t=\frac{s(9-s)^2}{(s+3)^3}
\end{gather*}
and the relation
\begin{gather*}
\tilde{z}(t)=\left(1+\frac{s}{3}\right)^{-2}z(s).
\end{gather*}
Then $\tilde{z}(t)$ satisf\/ies (\ref{hypergeom}) with $(a,b,c)=\left(-\frac{2}{3}, -\frac{1}{3}, \frac{1}{2}\right)$ if\/f $z(s)$ satisf\/ies (\ref{hypergeom}) with $(a,b,c)=\left(-2, -\frac{1}{2}, \frac{1}{2}\right)$. Now
\begin{gather*}
{}_2{F}_1\left(-\frac{2}{3},-\frac{1}{3};\frac{1}{2};t\right)
=(1-t)^{\frac{2}{3}}\, {}_2{F}_1\left(-\frac{2}{3},\frac{5}{6};\frac{1}{2};\frac{t}{t-1}\right)
\end{gather*}
while the map
$s \mapsto 1-s$ takes the dif\/ferential equation (\ref{hypergeom}) with $(a,b,c)=\big({-}\frac{2}{3}, \frac{5}{6}, \frac{1}{2}\big)$ to $(a,b,c)=\big({-}\frac{2}{3}, \frac{5}{6}, \frac{2}{3}\big)$.
Moreover, we have
\begin{gather*}
{}_2{F}_1\left(-2,-\frac{1}{2};\frac{1}{2};1-s\right)
=\frac{1}{3}\big(8-4s-s^2\big)
=\frac{8}{3}\,{}_2{F}_1\left(-2,-\frac{1}{2};-2;s\right),
\end{gather*}
so that all together the composition gives the map
\begin{gather*}
s \mapsto \tilde t=-\frac{(s-4)^3}{27 s^2}
\end{gather*}
and the relation
\begin{gather*}
\tilde{z}(\tilde t)=s^{-\frac{4}{3}}z(s).
\end{gather*}
Thus $\tilde{z}(\tilde t)$ satisf\/ies~(\ref{hypergeom}) with $(a,b,c)=\big({-}\frac{2}{3}, \frac{5}{6}, \frac{2}{3}\big)$ if\/f $z(s)$ satisf\/ies~(\ref{hypergeom}) with $(a,b,c)=\big({-}2$, $-\frac{1}{2}, -2\big)$.

To obtain the solution given in the third row from those in the f\/irst row requires a transformation of degree $4$ (see~\cite[formulas~(25)--(27)]{hyper} and \cite[equation~(131)]{goursat}).
Consider the map
\begin{gather*}
s \mapsto t=-\frac{s(s+8)^3}{64(1-s)^3}
\end{gather*}
and the relation
\begin{gather*}
\tilde{z}(t)=(1-s)^{\frac{5}{2}}z(s).
\end{gather*}
Then $\tilde{z}(t)$ satisf\/ies (\ref{hypergeom}) with $(a,b,c)=\big(\frac{5}{6},\!{-}\frac{2}{3},\! \frac{2}{3}\big)$ if\/f $z(s)$ satisf\/ies (\ref{hypergeom}) with $(a,b,c)=\big(\frac{11}{6}, \!\frac{10}{3}, \!\frac{2}{3}\big)$. Finally, we use the Euler transformation,{\samepage
\begin{gather*}
{}_2{F}_1\left(-\frac{7}{6},-\frac{8}{3};\frac{2}{3};s\right)=(1-s)^{\frac{9}{2}}\, {}_2{F}_1\left(\frac{11}{6},\frac{10}{3};\frac{2}{3};s\right)
\end{gather*}
to obtain the solution in the $3^{\rm rd}$ row.}

There is furthermore a degree 6 transformation relating the solution in the f\/irst row of Table~\ref{table1} to the solution in the f\/irst row of Table \ref{table2} (cf.~\cite[equation~(28)]{hyper}, \cite[equation~(134)]{goursat} and~\cite[Table~1]{acht}).
Consider the map
\begin{gather*}
s \mapsto t=\frac{27s^2(s-1)^2}{4(s^2-s+1)^3}
\end{gather*}
and the relation
\begin{gather*}
\tilde{z}(t)=(1-s+s^2)^{-2}z(s).
\end{gather*}
Then $\tilde{z}(t)$ satisf\/ies~(\ref{hypergeom}) with $(a,b,c)=\big({-}\frac{1}{3}, -\frac{2}{3}, -\frac{1}{2}\big)$ if\/f $z(s)$ satisf\/ies~(\ref{hypergeom}) with $(a,b,c)= (-4, -1, -2 )$. Furthermore
$s \mapsto 1-s$ takes the solution to~(\ref{hypergeom}) with
$(a,b,c)=\big({-}\frac{1}{3}, -\frac{2}{3}, -\frac{1}{2}\big)$ to the solution with $(a,b,c)=\big({-}\frac{1}{3}, -\frac{2}{3}, \frac{1}{2}\big)$
and an Euler transformation followed by $s \mapsto 1-s$ again takes this to the solution with
$(a,b,c)=\big(\frac{5}{6}, -\frac{2}{3}, \frac{2}{3}\big)$ as given in row~1 of Table~\ref{table2}.

Finally, consider the parametrisation given by
\begin{gather}\label{y411}
y=-4\Omega_1-\Omega_2-\Omega_3=\frac{\der}{\der q}\log \frac{\dot{s}^3}{(s-1)^\frac{5}{2}s^{\frac{5}{2}}}.
\end{gather}
We f\/ind that $(\alpha,\beta,\gamma)$ is one of
\begin{gather*}
\left(\frac{4}{k},\frac{1}{k},\frac{1}{k}\right), \qquad \left(\frac{2}{3},\frac{1}{k},\frac{1}{k}\right),\qquad \left(\frac{2}{3},\frac{1}{6},\frac{1}{6}\right).
\end{gather*}
For $(\alpha,\beta,\gamma)=\left(\frac{4}{k},\frac{1}{k},\frac{1}{k}\right)$,
we f\/ind that
$(a,b,c)=\big(\frac{k-6}{2k},\frac{k+2}{2k},\frac{k-1}{k}\big)$. For $k=\frac{2}{3}$, we obtain $(a,b,c)=(-4,2,-\frac{1}{2})$. Let us relate the solution to the dif\/ferential equation (\ref{hypergeom}) with $(a,b,c)=(-4,2,-\frac{1}{2})$ to the solution given in the second row of Table \ref{table2}.
For $(\alpha,\beta,\gamma)=\big(\frac{1}{2},\frac{1}{k},\frac{2}{k}\big)$,
we f\/ind that
$(a,b,c)=\big(\frac{k-6}{4k},\frac{3k-6}{4k},\frac{k-1}{k}\big)$.
We have
\begin{gather*}
{}_2{F}_1\left(\frac{1}{2}-\frac{3}{k},\frac{1}{2}+\frac{1}{k};1-\frac{1}{k};s\right)
 ={}_2{F}_1\left(\frac{1}{4}-\frac{3}{2k},\frac{1}{4}+\frac{1}{2k};1-\frac{1}{k};4s(1-s)\right)\\
\qquad {} =(1-4s(1-s))^{-\frac{1}{4}+\frac{3}{2k}}{}_2F{}_1\left(\frac{1}{4}
-\frac{3}{2k},\frac{3}{4}-\frac{3}{2k};1-\frac{1}{k};\frac{4s(1-s)}{4s(1-s)-1}\right)
\end{gather*}
and thus for $k=\frac{2}{3}$, we have
\begin{gather*}
{}_2{F}_1\left(-4,2;-\frac{1}{2};s\right)=(1-4s(1-s))^{2}\, {}_2{F}_1\left(-2,-\frac{3}{2};-\frac{1}{2};\frac{4s(1-s)}{4s(1-s)-1}\right).
\end{gather*}
Again $s \mapsto 1-s$ brings the solution to (\ref{hypergeom}) with $(a,b,c)=\big({-}2,-\frac{3}{2},-\frac{1}{2}\big)$ to the solution with $(a,b,c)=\left(-2,-\frac{3}{2},-2\right)$ and an Euler transform gives the solution with $(a,b,c)=\big({-}2,-\frac{1}{2},-2\big)$.

For $(\alpha,\beta,\gamma)=\big(\frac{2}{3},\frac{1}{k},\frac{1}{k}\big)$,
we f\/ind that
$(a,b,c)=\big(\frac{k-6}{6k},\frac{5k-6}{6k},\frac{k-1}{k}\big)$. For $k=\frac{2}{3}$, this gives
$(a,b,c)=\big({-}\frac{4}{3},-\frac{2}{3},-\frac{1}{2}\big)$.
We have a degree 2 transformation of the solution to~(\ref{hypergeom}) with this value to the solution in the f\/irst row of Table~\ref{table2} given by the map
\begin{gather*}
s \mapsto t=-\frac{1}{4s(s-1)}
\end{gather*}
and the relation
\begin{gather*}
\tilde z(t)=\frac{1}{4}(s(1-s))^{-\frac{2}{3}}z(s).
\end{gather*}
We have $\tilde z(t)$ satisfying (\ref{hypergeom}) with $(a,b,c)=\big({-}\frac{2}{3},\frac{5}{6},\frac{2}{3}\big)$
if\/f~$z(s)$ satisf\/ies (\ref{hypergeom}) with $(a,b,c)=\big({-}\frac{4}{3},-\frac{2}{3},-\frac{1}{2}\big)$. Let us summarise the solutions to~(\ref{ode1}) with parametrisation given by~(\ref{y411}) in Table~\ref{table3}.

 \begin{table}[h]\centering \caption{} \label{table3}\vspace{1mm}
 \begin{tabular}{|c|c|c|} \hline
 $(\alpha,\beta,\gamma)$&$(a,b,c)$ & General solution to $u''+\frac{1}{4}V(s) u=0$ \tsep{2pt}\bsep{2pt}\\ \hline\hline
$\left(6, \frac{3}{2}, \frac{3}{2}\right)$ & $\left({-}4, 2, -\frac{1}{2}\right)$ & $c_1 (2s-1)(s(s-1))^{\frac{5}{4}}+c_2 \frac{128s^4-256s^3+144s^2-16s-1}{(s(s-1))^{\frac{1}{4}}}$ \tsep{5pt}\bsep{6pt}\\ \hline
$\left(\frac{2}{3 }, \frac{3}{2},\frac{3}{2}\right)$ & $\left({-}\frac{4}{3}, -\frac{2}{3}, -\frac{1}{2}\right)$ &$c_1 (s(s-1))^{\frac{5}{4}}\,{}_2{F}_1\left(\frac{5}{3},\frac{7}{3};\frac{5}{2};s\right) +c_2 \frac{(s-1)^{\frac{5}{4}}}{s^{\frac{1}{4}}}\, {}_2{F}_1\left(\frac{1}{6},\frac{5}{6};-\frac{1}{2};s\right)$ \tsep{7pt}\bsep{6pt} \\ \hline
$\left(\frac{2}{3}, \frac{1}{6},\frac{1}{6}\right)$ & $\left(0, \frac{2}{3}, \frac{5}{6}\right)$& $c_1 (s(s-1))^{\frac{7}{12}}\, {}_2{F}_1\left(\frac{1}{3},1;\frac{7}{6};s\right)+c_2 s^{\frac{5}{12}}(1-s)^{\frac{5}{12}}$ \tsep{5pt}\bsep{2pt} \\ \hline
\end{tabular}
 \end{table}

We also note that we have
\begin{gather*}
{}_2{F}_1\left(-4,2;-\frac{1}{2};s\right) =-\frac{128}{5}P^{\left(-\frac{3}{2},-\frac{3}{2}\right)}_{4}(1-2s)=-128s^4+256s^3-144s^2+16s+1,\\
{}_2F_1\left(-\frac{4}{3},-\frac{2}{3};-\frac{1}{2};s\right) =-\frac{16}{27}\frac{\pi^2 2^{\frac{2}{3}}}{\Gamma\left(\frac{2}{3}\right)^3}P^{\left(-\frac{3}{2},-\frac{3}{2}\right)}_{\frac{4}{3}}(1-2s),\\
{}_2F_1\left(0,\frac{2}{3};\frac{5}{6};s\right) =1.
\end{gather*}
Comparing the parametrisation (\ref{y411}) with solutions of the form (\ref{original}), we obtain $(a,b,c)=\big(0, \frac{2}{3}, \frac{5}{6}\big)$.
For this solution given in the third row, we obtain solutions of the form $y(q)=-\frac{6}{q}$. Up to fractional linear transformations in the variable $s$, we have $7$ classes of solutions to~(\ref{3rds}) determined by the solutions to~(\ref{schwarzian}). They are given by the solutions with $(\alpha,\beta,\gamma)$ either $(3, 3, 3)$ or
$\big(3, \frac{1}{3}, \frac{1}{3}\big)$
for $I(q)$ given by~(\ref{yes}), $(\alpha,\beta,\gamma)$ one of $\big(\frac{3}{2 }, \frac{1}{3},\frac{1}{2}\big)$, $\big(\frac{3}{2}, \frac{1}{3},\frac{9}{2}\big)$ or $\big(\frac{3}{2}, 3,\frac{1}{2}\big)$
for $I(q)$ given by~(\ref{schwarz}) and $(\alpha,\beta,\gamma)$ either $\big(6,\frac{3}{2 }, \frac{3}{2}\big)$ or $\big(\frac{2}{3}, \frac{3}{2},\frac{3}{2}\big)$
for~$I(q)$ given by~(\ref{y411}).

When $\alpha=\beta=\gamma=0$, we have $\tau=0$. The f\/irst order system (\ref{chazy1}) is the classical Darboux--Halphen system and $y=-2(\Omega_1+\Omega_2+\Omega_3)$ satisf\/ies Chazy's equation~(\ref{chazy}). This system arises as the anti-self-dual Ricci-f\/lat equations for Bianchi-IX metrics (see~\cite{mono,tod2}).

\section[Examples of f\/lat $(2,3,5)$-distributions]{Examples of f\/lat $\boldsymbol{(2,3,5)}$-distributions}\label{examples}

In \cite{conf}, Nurowski associated to $(2,3,5)$-distributions a conformal class of metrics of signatu\-re~$(2,3)$. The fundamental curvature invariant of $(2,3,5)$-distributions appears as the Weyl tensor of Nurowski's metric. For $(2,3,5)$-distributions~$\mathcal{D}_{F(q)}$ determined by a single function~$F(q)$, the metric is described in~\cite{annur,conf}. The distribution $\mathcal{D}_{F(q)}$ on $M_{xyzpq}$ is encoded by the annihilator of the three 1-forms
\begin{gather}\label{coframe}
\omega_1 =\der y-p \der x,\qquad
\omega_2 =\der p-q \der x,\qquad
\omega_3 =\der z-F(q) \der x,
\end{gather}
and supplemented by the 1-forms
\begin{gather*}
\omega_4=\der q, \qquad \omega_5=\der x.
\end{gather*}
The coframe on $M_{xyzpq}$ is given by
\begin{gather}
\theta^1 =\omega_1-\frac{1}{F''}\left(F'\omega_2-\omega_3\right),\qquad
\theta^2 =\frac{1}{F''}\left(F'\omega_2-\omega_3\right),\nonumber\\
\theta^3 =\left(1-\frac{F' F^{(3)} }{4 (F'')^2 }\right)\omega_2+\frac{F^{(3)}}{4 (F'')^2}\omega_3,\nonumber\\
\theta^4 =\left(\frac{7 (F^{(3)})^2-4 F''F^{(4)}}{40(F'')^3}\right)\left(F' \omega_2-\omega_3\right)+\omega_4-\omega_5,\qquad
\theta^5 =-\omega_4,\label{coframe2}
\end{gather}
(cf.~\cite{annur}) and Nurowski's metric
\begin{gather}\label{nurmetric}
g_{\mathcal{D}_{F(q)}}=2\theta^1 \theta^5-2\theta^2 \theta^4+\frac{4}{3}\theta^3 \theta^3
\end{gather}
has vanishing Weyl tensor (and hence conformally f\/lat) if\/f $\mathcal{D}_{F(q)}$ has the split real form of~$G_2$ as its group of local symmetries. There is a more elegant way to present Nurowski's metric for the distribution~$\mathcal{D}_{F(q)}$. Equivalently, we can encode the distribution by
\begin{gather}
\tilde\omega_1 =\omega_1=\der y-p \der x,\nonumber\\
\tilde\omega_2 =\frac{1}{F''}(F'\omega_2-\omega_3)=\frac{1}{F''}\left(F'(\der p-q \der x)-(\der z-F(q) \der x)\right),\nonumber\\
\tilde\omega_3 =\omega_2=\der p-q \der x,\label{coframe3}
\end{gather}
with $\tilde\omega_4=\omega_4$ and $\tilde \omega_5=\omega_5$. The coframe is then given by
\begin{gather*}
\theta^1 =\tilde \omega_1-\tilde \omega_2, \qquad
\theta^2=\tilde \omega_2,\qquad
\theta^3=\tilde \omega_3-\frac{F^{(3)}}{4F'' }\tilde \omega_2,\\
\theta^4 =\left(\frac{7 (F^{(3)})^2-4 F''F^{(4)}}{40(F'')^2}\right)\tilde \omega_2+\tilde \omega_4-\tilde \omega_5, \qquad
\theta^5=-\tilde \omega_4.
\end{gather*}
Let us take
\begin{gather*}
F(q)=\iint e^{\frac{1}{2}\int I(q) \der q} \der q \der q,
\end{gather*}
as in (\ref{intp}). This gives $F''= e^{\frac{1}{2} \int I(q) \der q}$. Nurowski's metric now has a very simple form given by
\begin{gather}
g_{\mathcal{D}_{F(q)}} =2\theta^1 \theta^5-2\theta^2 \theta^4+\frac{4}{3}\theta^3 \theta^3\nonumber\\
\hphantom{g_{\mathcal{D}_{F(q)}}}{} =2\tilde \omega_2 \tilde \omega_5-2\tilde \omega_1 \tilde \omega_4+\frac{4}{3}(\tilde\omega_3)^2-\frac{I}{3}\tilde \omega_2 \tilde \omega_3+\frac{1}{10}\left(I'-\frac{I^2}{6}\right)(\tilde \omega_2)^2.\label{nurmetric1}
\end{gather}
The Ricci tensor of the above metric is
\begin{align*}
R_{ab}\theta^a\theta^b=\frac{9}{120}\big(6I'-I^2\big)(\tilde \omega_4)^2.
\end{align*}
We can consider conformal rescalings of the metric such that $\hat g_{\mathcal{D}_{F(q)}}=\Omega^2 g_{\mathcal{D}_{F(q)}}$ is Ricci f\/lat. It turns out that if we take $\Omega=\nu(q)^{-1}>0$, then the Ricci tensor of the rescaled metric $\hat g_{\mathcal{D}_{F(q)}}$ is given by
\begin{align}\label{ricciflat}
R_{ab}\theta^a\theta^b=\frac{3}{40\nu}\big(40 \nu''+\big(6I'-I^2\big)\nu\big)(\tilde \omega_4)^2,
\end{align}
so the appropriate conformal scale $\nu(q)$ can be found by solving the dif\/ferential equation in~(\ref{ricciflat}) (cf.\ \cite[Proposition~35]{tw13}).
In the f\/irst part of this section we consider the conformally f\/lat metrics~(\ref{nurmetric1}) obtained by solving~(\ref{3rds}) using the solution~(\ref{original}). Next, we then consider the solutions obtained from dif\/ferent parametrisations of the generalised Chazy equation given by~(\ref{yes}),~(\ref{schwarz}) and~(\ref{y411}). We also consider conformally f\/lat metrics obtained from solving the Legendre transform of~(\ref{3rds}). This involves computing the coframe for the metric under the Legendre transform. Finally, we consider the metrics obtain from Chazy's solutions given by~(\ref{nsol}). The metrics associated to $(2,3,5)$-distributions $\mathcal{D}_{F(q)}$ of the form $F(q)=q^m$ where $m \in \big\{{-}1,\frac{1}{3},\frac{2}{3},2\big\}$ are given in~\cite{new}.

\subsection{Chazy's solution}
In order to express Nurowski's metric associated to f\/lat $(2,3,5)$-distributions obtained from solving~(\ref{3rds}), we have to switch independent variable $s$ and dependent variable~$q$. In other words we pass to coordinates $(x,y,z,p,s)$ with $q(s)=\frac{z_2(s)}{z_1(s)}$ where~$z_1(s)$, $z_2(s)$ are given in Corollary~\ref{I235}.

Let us f\/irst consider solutions of the form
\begin{gather*}
I(q(s))=6 \frac{\der}{\der q} \log z_1(s)
\end{gather*}
as in Corollary \ref{I235}.
Observe that
\begin{gather*}
\int I(q) \der q=6 \log z_1
\end{gather*}
and so
\begin{gather*}
F(q)=\iint {\rm e}^{\frac{1}{2}\int I(q) \der q } \der q \der q=\iint (z_1)^3 \der q \der q.
\end{gather*}
For this parametrisation we have
\begin{gather*}
\der q=\frac{z_1 \dot{z}_2-z_2 \dot{z}_1}{(z_1)^2}\der s,
\end{gather*}
so that
\begin{gather*}
F(q(s))=\int \bigg(\int z_1(z_1 \dot{z}_2-z_2 \dot{z}_1 )\der s\bigg)\frac{z_1 \dot{z}_2-z_2 \dot{z}_1}{(z_1)^2}\der s.
\end{gather*}
Let us denote
\begin{gather*}
K(s)=\int z_1(z_1 \dot{z}_2-z_2 \dot{z}_1 )\der s.
\end{gather*}

\begin{Theorem}\label{flat235}
Let $z_1(s)$, $z_2(s)$ be two linearly independent solutions to~\eqref{hypergeom} with $(a,b,c)$ given by one of the list in Corollary~{\rm \ref{I235}}.
Let $\mathcal{D}_C$ denote the $(2,3,5)$-distribution on $M_{xyzps}$ associated to the annihilator of
\begin{gather*}
\tilde \omega_1 =\der y-p \der x,\\
\tilde \omega_2 =\frac{K(s)}{(z_1)^3}\left(\der p-\frac{z_2}{z_1}\der x\right)-\frac{1}{(z_1)^3}\left(\der z-\bigg(\int K(s)\frac{z_1 \dot{z}_2-z_2 \dot{z}_1}{(z_1)^2}\der s\bigg) \der x\right),\\
\tilde \omega_3 =\der p-\frac{z_2}{z_1}\der x.
\end{gather*}
Supplement by the $1$-forms
\begin{gather*}
\tilde \omega_4=\frac{z_1 \dot{z}_2-z_2 \dot{z}_1}{(z_1)^2}\der s, \qquad \tilde \omega_5=\der x.
\end{gather*}
Then Nurowski's metric~\eqref{nurmetric1}
\begin{align*}
g_{\mathcal{D}_C}
&=2\tilde \omega_2 \tilde \omega_5-2\tilde \omega_1 \tilde \omega_4+\frac{4}{3}(\tilde\omega_3)^2-\frac{2 z_1 \dot{z}_1}{z_1 \dot{z}_2-z_2 \dot{z}_1}\tilde \omega_2 \tilde \omega_3\\
&\quad+\frac{3}{5}\left(\frac{ (z_1)^3 \ddot{z}_1}{(z_1 \dot{z}_2-z_2 \dot{z}_1)^2}-\frac{(z_1)^3 \dot{z}_2(z_1\ddot{z}_2-z_2\ddot{z}_1)}{(z_1 \dot{z}_2-z_2 \dot{z}_1)^3}\right)(\tilde \omega_2)^2
\end{align*}
has vanishing Weyl tensor $($and hence conformally flat$)$ and $\mathcal{D}_C$ has the split real form of~$G_2$ as its group of local symmetries.
\end{Theorem}

\looseness=-1
Let us provide an explicit example given by Corollary~\ref{I235}.
It turns out for the values of $(a,b,c)$ obtained in Corollary~\ref{I235}, the solutions can be given by elementary functions.
For $(a,b,c)=\big({-}\frac{2}{3},\frac{5}{6},\frac{1}{2}\big)$, the solutions to the hypergeometric dif\/ferential equation~(\ref{hypergeom}) are given by
\begin{gather*}
z(s)=\mu(s-1)^{\frac{1}{6}}P^{\frac{1}{3}}_{1}(\sqrt{s})+\nu(s-1)^{\frac{1}{6}}Q^{\frac{1}{3}}_{1}(\sqrt{s}),
\end{gather*}
where $P^m_{\ell}$ and $Q^m_{\ell}$ are the associated Legendre functions.
This suggest passing further to the variable $r=\sqrt{s}$, in which case (\ref{hypergeom}) with $(a,b,c)=\big({-}\frac{2}{3},\frac{5}{6},\frac{1}{2}\big)$ becomes
\begin{gather}\label{trans}
\frac{1}{4}(1-r^2) z''(r)-\frac{1}{3}r z'(r)+\frac{5}{9}z(r)=0.
\end{gather}
The general solution is now given by the elementary functions
\begin{gather*}
z(r)=c_1 (r-1)^{\frac{1}{3}}(3 r+1)+c_2 (r+1)^{\frac{1}{3}}(3r-1).
\end{gather*}
There is thus a 4-dimensional space of solutions given by
\begin{gather*}
z_1(r) =c_1 (r-1)^{\frac{1}{3}}(3 r+1)+c_2 (r+1)^{\frac{1}{3}}(3r-1),\\
z_2(r) =c_3 (r-1)^{\frac{1}{3}}(3 r+1)+c_4 (r+1)^{\frac{1}{3}}(3r-1),
\end{gather*}
where $c_1 c_4-c_2 c_3 \neq 0$.
We also note that in the case $(a,b,c)=\big({-}\frac{2}{3},\frac{5}{6},\frac{2}{3}\big)$, the change of variable $r=\sqrt{1-s}$ brings (\ref{hypergeom}) to (\ref{trans}), so that the two hypergeometric ODEs with $(a,b,c)=\big({-}\frac{2}{3},\frac{5}{6},\frac{1}{2}\big)$ and $\big({-}\frac{2}{3},\frac{5}{6},\frac{2}{3}\big)$ can be brought to the same equation~(\ref{trans}) by a coordinate transformation.

Moreover, since
\begin{gather*}
{}_2{F}_1\left(-\frac{2}{3},\frac{5}{6};\frac{2}{3};4s(1-s)\right)={}_2{F}_1\left(-\frac{4}{3},\frac{5}{3};\frac{2}{3};s\right),
\end{gather*}
if we take $r=\sqrt{1-4s(1-s)}=2s-1$, we can pass from~(\ref{hypergeom}) with $(a,b,c)=\big({-}\frac{4}{3},\frac{5}{3};\frac{2}{3}\big)$ to equation~(\ref{trans}). Also compare with the second row of Table~\ref{table1}.

We now pass to coordinates on $M_{xyzpr}$ and take for a simple example
\begin{gather*}
z_1(r)=c_1 (r-1)^{\frac{1}{3}}(3 r+1),\qquad z_2(r)=c_2 (r+1)^{\frac{1}{3}}(3r-1)
\end{gather*}
where~$c_1$,~$c_2$ are non-zero constants.
We also make use of
\begin{gather*}
\der s=2 r \der r.
\end{gather*}
Here we have $
K(r)=-16(c_1)^2 c_2(r-1)^{\frac{2}{3}}(r+1)^{\frac{1}{3}}$.
We have

\begin{Proposition}\label{flatexample}
Let $c_1$, $c_2$ be non-zero constants.
Let $\mathcal{D}_0$ denote the $(2,3,5)$-distribution on $M_{xyzpr}$ associated to the annihilator of
\begin{gather*}
\tilde \omega_1 =\der y-p \der x,\\
\tilde \omega_2 =\frac{-16c_2(r+1)^{\frac{1}{3}}}{c_1(r-1)^{\frac{1}{3}}(3 r+1)^3}\der p
-\frac{1}{(c_1)^3(r-1)(3 r+1)^3}\der z+\frac{16(c_2)^2(r+1)^{\frac{2}{3}}}{(c_1)^2(r-1)^{\frac{2}{3}}(3 r+1)^3}\der x,\\
\tilde \omega_3 =\der p-\frac{c_2 (r+1)^{\frac{1}{3}}(3r-1)}{c_1 (r-1)^{\frac{1}{3}}(3 r+1)}\der x,
\end{gather*}
and supplemented by the $1$-forms
\begin{gather*}
\tilde \omega_4=-\frac{16 c_2}{3c_1 (3r+1)^2 (r-1)^{\frac{4}{3}}(r+1)^{\frac{2}{3}}}\der r, \qquad \tilde \omega_5=\der x.
\end{gather*}
Then Nurowski's metric~\eqref{nurmetric1}
\begin{gather*}
g_{\mathcal{D}_0}
 =2\tilde \omega_2 \tilde \omega_5-2\tilde \omega_1 \tilde \omega_4+\frac{4}{3}(\tilde\omega_3)^2+\frac{c_1(3r-2)(3r+1)(r-1)^{\frac{1}{3}}(r+1)^{\frac{2}{3}}}{2 c_2}\tilde \omega_2 \tilde \omega_3\\
\hphantom{g_{\mathcal{D}_0}=}{}
+\frac{3(c_1)^2 (r-1)^{\frac{5}{3}} (3r+1)^4 (r+1)^{\frac{1}{3}}}{64 (c_2)^2}(\tilde \omega_2)^2
\end{gather*}
has vanishing Weyl tensor $($and hence conformally flat$)$ and $\mathcal{D}_0$ has the split real form of~$G_2$ as its group of local symmetries. The Ricci tensor for this metric is
\begin{gather*}
R_{ab}\theta^a\theta^b=\frac{6}{r^2-1}\der r \der r.
\end{gather*}
Rescaling this metric by
\begin{gather*}
\Omega=\frac{1}{\nu}=\frac{4}{3}\frac{(3r+1)(r-1)^{\frac{1}{3}}}{a_1(r-1)^{\frac{1}{3}}-a_2(r+1)^{\frac{1}{3}}},
\end{gather*}
where $a_1$ and $a_2$ are constants,
the conformally rescaled metric $\hat g_{\mathcal{D}_0}=\Omega^2 g_{\mathcal{D}_0}$ given by
\begin{gather*}
\hat g_{\mathcal{D}_0} =\frac{16(3r+1)^2(r-1)^{\frac{2}{3}}}{9(a_1(r-1)^{\frac{1}{3}}-a_2(r+1)^{\frac{1}{3}})^2}\left(2\tilde \omega_2 \tilde \omega_5-2\tilde \omega_1 \tilde \omega_4+\frac{4}{3}(\tilde\omega_3)^2\right.\\
\left.\hphantom{\hat g_{\mathcal{D}_0} =}{}
+\frac{c_1(3r-2)(3r+1)(r-1)^{\frac{1}{3}}(r+1)^{\frac{2}{3}}}{2 c_2}\tilde \omega_2 \tilde \omega_3+\frac{3(c_1)^2 (r-1)^{\frac{5}{3}} (3r+1)^4 (r+1)^{\frac{1}{3}}}{64 (c_2)^2}(\tilde \omega_2)^2\!\right)\!
\end{gather*}
is both Ricci-flat and conformally flat.
\end{Proposition}

\subsection{Other parametrisations of the generalised Chazy equation}\label{opara}

Instead of choosing $I(q)=6\frac{\der}{\der q}\log z_1$, we can
consider the parametrisation
\begin{gather*}
I(q)=\frac{\der}{\der q}\log \frac{\dot{s}^3}{s^2(s-1)^2}
\end{gather*}
given in~(\ref{yes}).
In this case
\begin{gather*}
F(q)=\iint {\rm e}^{\frac{1}{2}\int I(q) \der q } \der q \der q=\iint \frac{\dot{s}^{\frac{3}{2}}}{s(s-1)} \der q \der q
\end{gather*}
and the corresponding metric associated to the $(2,3,5)$-distribution $\mathcal{D}_{F(q)}$ gives the following

\begin{Theorem}\label{flat235s}
Let $s(q)$ be a solution to \eqref{schwarzian} with $(\alpha,\beta,\gamma)$ given by one of
\begin{gather*}
 (3, 3, 3 ),\qquad
\left(3, \frac{1}{3}, \frac{1}{3}\right),\qquad
\left(\frac{1}{3 }, 3,\frac{1}{3}\right),\qquad
\left(\frac{1}{3}, \frac{1}{3},3\right).
\end{gather*}
Let $\mathcal{D}_{s}$ denote the $(2,3,5)$-distribution on $M_{xyzpq}$ associated to the annihilator of
\begin{gather*}
\omega_1 =\der y-p \der x,\qquad
\omega_2 =\der p-q\der x,\qquad
\omega_3 =\der z-\bigg(\iint \frac{\dot{s}^{\frac{3}{2}}}{s(s-1)} \der q \der q\bigg) \der x.
\end{gather*}
Supplement by the $1$-forms
\begin{gather*}
\omega_4=\der q, \qquad \omega_5=\der x,
\end{gather*}
and take the coframe on $M_{xyzpq}$ to be given by $(\theta^1, \theta^2, \theta^3, \theta^4, \theta^5)$ as in~\eqref{coframe2} where
\begin{gather*}
F(q)=\iint \frac{\dot{s}^{\frac{3}{2}}}{s(s-1)} \der q \der q.
\end{gather*}
Then Nurowski's metric
\begin{gather*}
g_{\mathcal{D}_{s}}=2\theta^1 \theta^5-2\theta^2 \theta^4+\frac{4}{3}\theta^3 \theta^3
\end{gather*}
has vanishing Weyl tensor $($and hence conformally flat$)$ and~$\mathcal{D}_{s}$ has the split real form of~$G_2$ as its group of local symmetries.
\end{Theorem}

To obtain explicit examples, it is useful to switch the independent variable~$q$ and the dependent variable $s$. We pass to the variables $(x,y,z,p,s)$ with
$q=\frac{u_2(s)}{u_1(s)}$ where $u_1$ and $u_2$ are linearly independent solutions of~(\ref{ode1}) given in Table~\ref{table1}.
Note that up to fractional linear transformations in the variable~$s$, we only need to consider the solutions to (\ref{schwarzian}) with the values of $(\alpha,\beta,\gamma)$ given by either $(3, 3, 3)$ or $
\big(3, \frac{1}{3}, \frac{1}{3}\big)$ for the parametrisation given by~(\ref{yes}).
Note the symmetry permuting~$\beta$ and~$\gamma$.

A computation shows that $W(u_1,u_2)=u_1\dot{u}_2-\dot{u}_1u_2$ is constant, which we can normalise to set $W(u_1,u_2)=1$.
We have
\begin{gather*}
\der q=\frac{u_1 \dot{u}_2-u_2 \dot{u}_1}{(u_1)^2}\der s=\frac{1}{(u_1)^2}\der s
\end{gather*}
and
\begin{gather*}
\dot{s}=\frac{1}{q'(s)}=(u_1)^2,
\end{gather*}
so that
\begin{gather*}
F(q(s))=\int \bigg(\int \frac{(u_1)^3}{s(s-1)}\frac{1}{(u_1)^2}\der s\bigg)\frac{1}{(u_1)^2}\der s=\int \bigg(\int \frac{u_1}{s(s-1)}\der s\bigg)\frac{1}{(u_1)^2}\der s.
\end{gather*}
Let us denote
\begin{gather*}
K(s)=\int \frac{u_1}{s(s-1)}\der s.
\end{gather*}

\begin{Theorem}\label{flat235s1}
Let $u_1(s)$, $u_2(s)$ be two linearly independent solutions to \eqref{ode1} subject to the constraint $W(u_1,u_2)=1$ with $(\alpha,\beta,\gamma)$ given by Table~{\rm \ref{table1}}.
Let $\mathcal{D}_s$ denote the $(2,3,5)$-distribution on $M_{xyzps}$ associated to the annihilator of
\begin{gather*}
\omega_1 =\der y-p \der x,\qquad
\omega_2 =\der p-\frac{u_2(s)}{u_1(s)}\der x,\qquad
\omega_3 =\der z-\bigg(\int \frac{K(s)}{(u_1)^2}\der s\bigg) \der x.
\end{gather*}
We pass to the annihilator $1$-forms given by \eqref{coframe3} to obtain
\begin{gather*}
\tilde \omega_1=\omega_1=\der y-p \der x,\qquad \tilde \omega_2=\frac{s(s-1)}{(u_1)^3}(K \omega_2-\omega_3),\qquad
\tilde \omega_3=\omega_2=\der p-\frac{u_2(s)}{u_1(s)}\der x,
\end{gather*}
with
\begin{gather*}
\tilde \omega_4=\omega_4=\frac{1}{(u_1)^2}\der s, \qquad \tilde \omega_5=\omega_5=\der x.
\end{gather*}
Take the coframe on $M_{xyzps}$ to be given by
\begin{gather*}
\theta^1 =\tilde\omega_1-\tilde \omega_2,\qquad
\theta^2=\tilde \omega_2,\qquad \theta^3=\tilde \omega_3-\frac{(u_1)^2}{4s(s-1)}\left(3\frac{\dot{u}_1}{u_1}s(s-1)-(2s-1)\right)\tilde \omega_2,\\
\theta^4 =\frac{(u_1)^4}{40}\left(\frac{4s^2-4s-1}{s^2(s-1)^2}+3 V(s)-\frac{10(2s-1)}{s(s-1)}\frac{\dot{u}_1}{u_1}+15\left(\frac{\dot{u}_1}{u_1}\right)^2\right)\tilde \omega_2 +\tilde \omega_4-\tilde \omega_5,\\
\theta^5 =-\tilde\omega_4.
\end{gather*}
Then Nurowski's metric
\begin{gather*}
g_{\mathcal{D}_s} =2\theta^1 \theta^5-2\theta^2 \theta^4+\frac{4}{3}\theta^3 \theta^3\\
\hphantom{g_{\mathcal{D}_s}}{} =2 \tilde \omega_2 \tilde \omega_5-2 \tilde \omega_1 \tilde \omega_4+\frac{4}{3}(\tilde \omega_3)^2+\frac{2 u_1(3 s(1-s)u_1'+(2s-1) u_1)}{3s (s-1)}\tilde \omega_2 \tilde \omega_3\\
\hphantom{g_{\mathcal{D}_s} =}{} -\frac{(u_1)^4 (9 V(s) s^2(s-1)^2-8(s^2-s+1))}{60 (s-1)^2 s^2}(\tilde \omega_2)^2
\end{gather*}
has vanishing Weyl tensor $($and hence conformally flat$)$ and $\mathcal{D}_s$ has the split real form of~$G_2$ as its group of local symmetries.
\end{Theorem}

The analogous results hold for $I(q)$ given by the formulas in (\ref{schwarz}), (\ref{schwarz2}) and (\ref{y411}).
If we take
\begin{gather*}
I(q)=\frac{\der}{\der q}\log \frac{\dot{s}^3}{s^2(s-1)^{\frac{3}{2}}},
\end{gather*}
as in (\ref{schwarz}),
we have the corresponding $(2,3,5)$-distribution associated to
\begin{gather*}
F(q)=\iint \frac{\dot{s}^{\frac{3}{2}}}{s(s-1)^{\frac{3}{4}}} \der q \der q.
\end{gather*}

\begin{Theorem}\label{flat235j}
Let $s(q)$ be a solution to \eqref{schwarzian} with $(\alpha,\beta,\gamma)$ given by $\big(\frac{3}{2}, \frac{1}{3}, \frac{1}{2}\big)$, $\big(\frac{3}{2}, 3, \frac{1}{2}\big)$ or $\big(\frac{3}{2}, \frac{1}{3},\frac{9}{2}\big)$.
Let $\mathcal{D}_{s_1}$ denote the $(2,3,5)$-distribution on $M_{xyzpq}$ associated to the annihilator of
\begin{gather*}
\omega_1 =\der y-p \der x,\qquad
\omega_2 =\der p-q\der x,\qquad
\omega_3 =\der z-\bigg(\iint \frac{\dot{s}^{\frac{3}{2}}}{s(s-1)^{\frac{3}{4}}} \der q \der q\bigg) \der x.
\end{gather*}
Supplement by the $1$-forms
\begin{gather*}
\omega_4=\der q, \qquad \omega_5=\der x,
\end{gather*}
and take the coframe on $M_{xyzpq}$ to be given by $(\theta^1, \theta^2, \theta^3, \theta^4, \theta^5)$ as in~\eqref{coframe2}.
Then Nurowski's metric~\eqref{nurmetric}
has vanishing Weyl tensor $($and hence conformally flat$)$ and $\mathcal{D}_{s_1}$ has the split real form of~$G_2$ as its group of local symmetries.
\end{Theorem}

Similarly, for $I(q)$ given by~(\ref{y411}), that is to say
\begin{gather*}
I(q)=\frac{\der}{\der q}\log \frac{\dot{s}^3}{s^{\frac{5}{2}}(s-1)^{\frac{5}{2}}},
\end{gather*}
we have the corresponding $(2,3,5)$-distribution associated to
\begin{gather*}
F(q)=\iint \frac{\dot{s}^{\frac{3}{2}}}{s^{\frac{5}{4}}(s-1)^{\frac{5}{4}}} \der q \der q.
\end{gather*}

\begin{Theorem}\label{flat235k}
Let $s(q)$ be a solution to~\eqref{schwarzian} with $(\alpha,\beta,\gamma)$ given by
$\big(6, \frac{3}{2}, \frac{3}{2}\big)$, $
\big(\frac{2}{3}, \frac{3}{2}, \frac{3}{2}\big)$ or
$\big(\frac{2}{3 }, \frac{1}{6},\frac{1}{6}\big)$.
Let $\mathcal{D}_{s_2}$ denote the $(2,3,5)$-distribution on $M_{xyzpq}$ associated to the annihilator of
\begin{gather*}
\omega_1 =\der y-p \der x,\qquad
\omega_2 =\der p-q\der x,\qquad
\omega_3 =\der z-\bigg(\iint \frac{\dot{s}^{\frac{3}{2}}}{s^{\frac{5}{4}}(s-1)^{\frac{5}{4}}} \der q \der q\bigg) \der x.
\end{gather*}
Supplement by the $1$-forms
\begin{align*}
\omega_4=\der q, \qquad \omega_5=\der x,
\end{align*}
and take the coframe on $M_{xyzpq}$ to be given by $(\theta^1, \theta^2, \theta^3, \theta^4, \theta^5)$ as in~\eqref{coframe2}.
Then Nurowski's metric~\eqref{nurmetric}
has vanishing Weyl tensor $($and hence conformally flat$)$ and~$\mathcal{D}_{s_2}$ has the split real form of~$G_2$ as its group of local symmetries.
\end{Theorem}

\subsection{Legendre transformed coframe}\label{ltc}
The Legendre transform of Proposition \ref{leg1} takes the 1-forms (\ref{coframe}) to
\begin{gather*}
\omega_1 =\der y-p \der x,\qquad
\omega_2 =\der p-H' \der x,\qquad
\omega_3 =\der z-(t H'-H) \der x,\\
\omega_4 =H'' \der t,\qquad
\omega_5 =\der x,
\end{gather*}
where $H=H(t)$ with $H'' \neq 0$ on $M_{xyzpt}$
and the coframe~(\ref{coframe2}) to
\begin{gather*}
\theta^1 =\omega_1-H''(t \omega_2-\omega_3),\qquad
\theta^2 = H''(t \omega_2-\omega_3),\qquad
\theta^3 =\left(1+t \frac{H'''}{4 H''}\right)\omega_2-\frac{H'''}{4 H''}\omega_3,\\
\theta^4 =\frac{4 H'' H''''-5(H''')^2}{40 (H'')^3}(t \omega_2-\omega_3)+\omega_4-\omega_5,\qquad
\theta^5 =-\omega_4.
\end{gather*}
Note that our $H(t)$ is related to $\Theta(x_5)$ of \cite{annur} via $\Theta_{55}=-H$, $t=x_5$.
The Nurowski metric
\begin{gather*}
g=2\theta^1 \theta^5-2 \theta^2 \theta^4+\frac{4}{3} \theta^3 \theta^3
\end{gather*}
has the only non-vanishing component of the Weyl tensor given by the left hand side term of the dual ODE (\ref{ds6}). This accounts for the appearance of equation (\ref{4thode}) in \cite{cartan1910}. The solutions of the dual generalised Chazy ODE (\ref{gc2}) with parameter $\pm \frac{3}{2}$ give us further examples of f\/lat $(2,3,5)$-distributions. We pass to $(x,y,z,p,s)$ as before, with $t(s)=\frac{w_2(s)}{w_1(s)}$ where $w_1(s)$, $w_2(s)$ are linearly independent solutions to~(\ref{hypergeom}) with $(a,b,c)$ one of
\begin{gather*}
\left(-\frac{1}{4 }, \frac{5}{12}, \frac{1}{2}\right),\qquad
\left(-\frac{1}{4}, \frac{5}{12}, \frac{2}{3}\right),\qquad
\left(-\frac{1}{2 }, \frac{5}{6},\frac{2}{3}\right).
\end{gather*}
Here we have taken $k=\frac{3}{2}$.
Note that the equations (\ref{hypergeom}) for $(a,b,c)=\big({-}\frac{1}{4 }, \frac{5}{12}, \frac{1}{2}\big)$ and $\big({-}\frac{1}{4}, \frac{5}{12}, \frac{2}{3}\big)$ are equivalent by a linear transformation and thus the solutions to each equation can be expressed as linear combinations of the other, while the solutions for $(a,b,c)=\big({-}\frac{1}{4}, \frac{5}{12}, \frac{2}{3}\big)$ and
$\big({-}\frac{1}{2 }, \frac{5}{6},\frac{2}{3}\big)$
are equivalent by a quadratic transformation as before. However, the author does not know if the solutions in these cases can be expressed by elementary functions. We consider once again solutions to~(\ref{gc2}) of the form $u(t)=6\frac{\der}{\der t}\log w_1$.
This gives
\begin{gather*}
H(t)=\iint (w_1)^4 \der t \der t, \qquad
H'(t)=\int (w_1)^4 \der t, \qquad H''(t)=(w_1)^4.
\end{gather*}
For this parametrisation we have
\begin{gather*}
\der t=\frac{w_1 \dot{w}_2-w_2 \dot{w}_1}{(w_1)^2} \der s
\end{gather*}
and so
\begin{gather*}
H=\iint (w_1)^2 (w_1 \dot{w}_2-w_2 \dot{w}_1) \der s \frac{w_1 \dot{w}_2-w_2 \dot{w}_1}{(w_1)^2} \der s,\qquad H'=\int (w_1)^2 (w_1 \dot{w}_2-w_2 \dot{w}_1) \der s.
\end{gather*}
We therefore obtain
\begin{gather*}
\omega_1 =\der y-p \der x,\qquad
\omega_2 =\der p-\int (w_1)^2 (w_1 \dot{w}_2-w_2 \dot{w}_1) \der s\der x,\\
\omega_3 =\der z-\left(\frac{w_2}{w_1} \int (w_1)^2 (w_1 \dot{w}_2-w_2 \dot{w}_1) \der s\right.\\
\left.\hphantom{\omega_3 =}{} -\iint (w_1)^2 (w_1 \dot{w}_2-w_2 \dot{w}_1) \der s \frac{w_1 \dot{w}_2-w_2 \dot{w}_1}{(w_1)^2} \der s\right) \der x,\\
\omega_4 =(w_1)^2 (w_1 \dot{w}_2-w_2 \dot{w}_1) \der s,\qquad
\omega_5 =\der x
\end{gather*}
and the adapted coframe for Nurowski's metric is
\begin{gather}
\theta^1 =\omega_1-(w_1)^4\left(\frac{w_2}{w_1} \omega_2-\omega_3\right),\qquad
\theta^2 = (w_1)^4\left(\frac{w_2}{w_1} \omega_2-\omega_3\right),\nonumber\\
\theta^3 =\left(1+\frac{w_2 \dot{w}_1}{w_1 \dot{w}_2-w_2\dot{w}_1}\right)\omega_2-\frac{w_1 \dot{w}_1}{w_1 \dot{w}_2-w_2\dot{w}_1}\omega_3,\nonumber\\
\theta^4
 =\frac{2(\ddot{w}_1\dot{w}_2-\ddot{w}_2 \dot{w}_1)}{5(w_1 \dot{w}_2-w_2 \dot{w}_1)^3}\left(\frac{w_2}{w_1} \omega_2-\omega_3\right)+\omega_4-\omega_5,\qquad
\theta^5 =-\omega_4. \label{lgco}
\end{gather}
Equivalently, we can take
\begin{align*}
\tilde\omega_1=\omega_1, \qquad \tilde\omega_2=(w_1)^4\left(\frac{w_2}{w_1} \omega_2-\omega_3\right),\qquad
\tilde\omega_3=\omega_2
\end{align*}
with $\tilde\omega_4=\omega_4$ and $\tilde \omega_5=\omega_5$. The coframe is then given by
\begin{gather}
\theta^1 =\tilde \omega_1-\tilde \omega_2, \qquad
\theta^2=\tilde \omega_2,\qquad
\theta^3=\tilde \omega_3+\frac{ \dot{w}_1}{(w_1)^3(w_1 \dot{w}_2-w_2\dot{w}_1)}\tilde \omega_2,\nonumber\\
\theta^4 =\frac{2(\ddot{w}_1\dot{w}_2-\ddot{w}_2 \dot{w}_1)}{5(w_1)^4(w_1 \dot{w}_2-w_2 \dot{w}_1)^3}\tilde \omega_2+\tilde \omega_4-\tilde \omega_5, \qquad
\theta^5=-\tilde \omega_4.\label{lgcoframe2}
\end{gather}

\begin{Proposition}\label{dualchazy}
The Nurowski metric
\begin{gather*}
g=2\theta^1 \theta^5-2 \theta^2 \theta^4+\frac{4}{3} \theta^3 \theta^3
\end{gather*}
given by the above coframe~\eqref{lgco} for $w_1(s)$, $w_2(s)$ linearly independent solutions to the hypergeometric differential equation~\eqref{hypergeom} with $(a,b,c)$ one of
\begin{gather*}
\left(-\frac{1}{4 }, \frac{5}{12}, \frac{1}{2}\right),\qquad
\left(-\frac{1}{4}, \frac{5}{12}, \frac{2}{3}\right),\qquad
\left(-\frac{1}{2 }, \frac{5}{6},\frac{2}{3}\right)
\end{gather*}
are all conformally flat. For each $(a,b,c)$ there is a $4$-dimensional family of solutions. Up to fractional linear transformation in the variable~$s$ there are~$2$ distinct classes given by the last two entries.
\end{Proposition}

In addition, the Legendre transformation of Lemma \ref{legendre} given by
\begin{gather*}
w_1 =(z_1)^{-\frac{3}{4}},\qquad
w_2 =(z_1)^{-\frac{3}{4}} \int z_1 (z_1 \dot{z}_2-z_2\dot{z}_1)\der s
\end{gather*}
takes the coframe (\ref{lgcoframe2}) to the coframe in Theorem~\ref{flat235} and conversely so.
The Legendre transform also applies to the coframes given in Theorems~\ref{flat235s}, \ref{flat235j} and~\ref{flat235k}. We also have the analogous results of Section~\ref{opara}. Up to fractional linear transformations in~$s$ we have~$7$ classes of solutions to the generalised Chazy equation~(\ref{gc2}) determined by $s(t)$ satisfying (\ref{schwarzian}). These are given by the parametrisations in Section~\ref{fos} and the corresponding values for $(\alpha,\beta,\gamma)$ can be computed for the parameter $k=\frac{3}{2}$. For the parametrisation
\begin{gather*}
H(t)=\iint \frac{\dot{s}^2}{s^{\frac{4}{3}}(s-1)^{\frac{4}{3}}} \der t \der t,
\end{gather*}
the values for $(\alpha,\beta,\gamma)$ are given by either $\big(\frac{4}{3}, \frac{4}{3}, \frac{4}{3}\big)$ or $\big(\frac{4}{3}, \frac{1}{3}, \frac{1}{3}\big)$.
For
\begin{gather*}
H(t)=\iint \frac{\dot{s}^2}{s^{\frac{4}{3}}(s-1)} \der t \der t,
\end{gather*}
$(\alpha,\beta,\gamma)$ takes the values of $\big(\frac{2}{3 }, \frac{1}{2},\frac{1}{3}\big)$,
$\big(\frac{2}{3}, \frac{1}{2},\frac{4}{3}\big)$ or
$\big(\frac{2}{3}, 2,\frac{1}{3}\big)$.
For
\begin{gather*}
H(t)=\iint \frac{\dot{s}^2}{s^{\frac{5}{3}}(s-1)^{\frac{5}{3}}} \der t \der t,
\end{gather*}
we obtain $\big(\frac{8}{3}, \frac{2}{3}, \frac{2}{3}\big)$ or
$\big(\frac{2}{3}, \frac{2}{3}, \frac{2}{3}\big)$.

The Legendre transform therefore provides seven further classes of f\/lat Nurowski metrics up to fractional linear transformations in~$s$.

\subsection{Additional examples}

The solution (\ref{nsol}) for $k=\pm\frac{2}{3}$ gives
$I(q)=-\frac{8}{3(q+C)}-\frac{10}{3(q+B)}$. Hence, the metric~(\ref{nurmetric1}) on $M_{xyzpq}$ given by
\begin{gather*}
g_{\mathcal{D}_{F(q)}} =2\tilde \omega_2 \tilde \omega_5-2\tilde \omega_1 \tilde \omega_4+\frac{4}{3}(\tilde\omega_3)^2\\
\hphantom{g_{\mathcal{D}_{F(q)}} =}{}
+\left(\frac{8}{9(q+C)}+\frac{10}{9(q+B)}\right)\tilde \omega_2 \tilde \omega_3+\frac{4(B-C)^2}{27(q+B)^2 (q+C)^2}(\tilde \omega_2)^2
\end{gather*}
is conformally f\/lat. In the dual coframe the f\/lat metric~(\ref{nurmetric1}) on $M_{xyzpt}$ is given by
\begin{gather}
g
 =2\tilde \omega_2 \tilde \omega_5-2\tilde \omega_1 \tilde \omega_4+\frac{4}{3}(\tilde\omega_3)^2\nonumber\\
\hphantom{g=}{}+\frac{4}{9}u(t) e^{\int -\frac{2}{3}u(t) \der t}\tilde \omega_2 \tilde \omega_3-\frac{2}{135}(9\dot u(t)-4 u(t)^2) e^{\int -\frac{4}{3}u(t) \der t}(\tilde \omega_2)^2,\label{gmetric}
\end{gather}
where $u(t)$ satisf\/ies the generalised Chazy equation (\ref{gc2}) with parameter $k=\pm\frac{3}{2}$. The solu\-tion~(\ref{nsol}) for $k=\pm\frac{3}{2}$ gives
$u(t)=-\frac{15}{4(t-a)}-\frac{9}{4(t-b)}$ and substituting this into~(\ref{gmetric}) gives the conformally f\/lat metric
\begin{gather*}
g =2\tilde \omega_2 \tilde \omega_5-2\tilde \omega_1 \tilde \omega_4+\frac{4}{3}(\tilde\omega_3)^2-\frac{256}{3}(8t-5b-3a)(t-a)^{\frac{3}{2}}(t-b)^{\frac{1}{2}}\tilde \omega_2 \tilde \omega_3\\
\hphantom{g=}{} +\frac{65536}{3}(t-b)^2(t-a)^3(4t-3a-b)(\tilde \omega_2)^2.
\end{gather*}

To summarise the results of this section, we f\/irst presented dif\/ferent examples of Nurowski metrics that are conformally f\/lat up to fractional linear transformation in the variable~$s$. Two examples are given in Theorem~\ref{flat235s}, three examples are given in Theorem~\ref{flat235j} and two more in Theorem~\ref{flat235k}. Seven additional examples are obtained from the Legendre transform as in Proposition~\ref{leg1}. We also have 2 additional examples from the solutions of the form (\ref{nsol}). Finally, there are examples associated to distributions of the form
$F(q)=q^m$, where $m \in \big\{{-}1,\frac{1}{3},\frac{2}{3},2\big\}$ and passing to the dual coframe, distributions of the form $H(t)=t^m$, where $m \in \big\{{-}2,-\frac{1}{2},\frac{1}{2},2\big\}$.

\section{An--Nurowski circle twistor bundle}\label{anexamples}

In \cite{annur0}, An and Nurowski showed how to associate to a split signature conformal structure $[g]$ on a 4-manifold $M^4$ a natural $(2,3,5)$-distribution. 4-dimensional split signature conformal structures admit real self-dual totally null 2-planes. The bundle of such 2-planes is a circle bundle over $M^4$ with f\/ibres $S^1$ \cite{annur0}. This is called the circle twistor bundle ${\mathbb T}(M^4)$ and it has a~rank~2 distribution given by lifting horizontally the null 2-planes on~$M^4$. This distribution is non-integrable, i.e., def\/ines a $(2,3,5)$-distribution whenever the self-dual part of the Weyl tensor of~$g$ on $M^4$ is non-vanishing. Moreover in~\cite{annur}, the authors presented split signature conformal structures on~$M^4$ that give rise to $(2,3,5)$-distributions of the form $\mathcal{D}_{F(q)}$ on ${\mathbb T}(M^4)$. Such split signature metrics are called Pleba\'nski's second heavenly metrics in~\cite{annur}.

Following \cite[Section~3]{annur}, we can f\/ind these metrics that have a f\/lat circle twistor bundle. Such circle twistor bundles have split $G_2$ as their group of symmetries.
Let $(w,x,y,z)$ be local coordinates on $M^4$. Let $\Theta=\Theta(w,x,y,z)$ be an arbitrary function of 4 variables (second heavenly function of Pleba\'nski).
Let $(e_i)$ be an orthonormal frame on $M^4$ and $(\theta^j)$ the dual coframe satisfying $\theta^j(e_i)=\delta^j{}_i$.
The split signature Pleba\'nski metric is given by
\begin{gather*}
g=g_{ij}\theta^i \otimes \theta^j=2\theta^1\theta^2+2\theta^3\theta^4,
\end{gather*}
where $\theta^i \theta^j=\frac{1}{2}\theta^i \otimes \theta^j+\frac{1}{2} \theta^j \otimes \theta^i$
and
\begin{gather*}
\theta^1 =\der x-\Theta_{yy}\der w+\Theta_{xy}\der z,\qquad
\theta^2 =\der w, \qquad
\theta^3 =\der y-\Theta_{xx} \der z+\Theta_{xy} \der w,\qquad
\theta^4 =\der z.
\end{gather*}
Hence \looseness=-1 $g_{12}=g_{34}=1$ and all other components are zero. Such split signature metrics admit a~real parallel spinor~\cite{nullkahler}.
A computation shows that the connection 1-forms we need are given by
\begin{gather*}
\Gamma^1{}_1 =-\Theta_{yyx}\theta^2+ \Theta_{yxx}\theta^4,\qquad
\Gamma^1{}_3 =-\Theta_{yyy}\theta^2+ \Theta_{yyx}\theta^4,\\
\Gamma^3{}_1 =\Theta_{yxx}\theta^2- \Theta_{xxx}\theta^4,\qquad
\Gamma^3{}_3 =\Theta_{yyx}\theta^2- \Theta_{yxx}\theta^4.
\end{gather*}
Using \cite{annur} and Nurowski's notes \cite{nurnotes}, we f\/ind that
the $(2,3,5)$-distribution on ${\mathbb T}(M^4)$ is annihilated by the following three $1$-forms:
\begin{align*}
\omega^3&=\der \xi+\Gamma^3{}_1+\big(\Gamma^3{}_3-\Gamma^1{}_1\big)\xi-\Gamma^1{}_3 \xi^2\\
&=\der \xi+\big(\Theta_{yxx}+2\Theta_{yyx}\xi+\Theta_{yyy}\xi^2\big)\theta^2-\big(\Theta_{xxx}+2\Theta_{yxx}\xi+H_{yyx}\xi^2\big)\theta^4\\
&=\der \xi+\big(\Theta_{yxx}+2\Theta_{yyx}\xi+\Theta_{yyy}\xi^2\big)\der w -\big(\Theta_{xxx}+2\Theta_{yxx}\xi+\Theta_{yyx}\xi^2\big)\der z
\end{align*}
and
\begin{align*}
\omega^4&=\xi \theta^4+\theta^2=\xi \der z+\der w,\\
\omega^5&=\theta^3-\xi \theta^1=\der y-\Theta_{xx}\der z+\Theta_{xy}\der w-\xi(\der x-\Theta_{yy}\der w+\Theta_{xy}\der z)\\
&=\der y-\xi \der x-(\Theta_{xx}+\xi \Theta_{xy})\der z+(\Theta_{xy}+\xi \Theta_{yy})\der w.
\end{align*}
The distribution is therefore annihilated by the 1-forms
\begin{gather*}
\tilde \omega^3 =\der \xi-\big(\Theta_{xxx}+3\Theta_{yxx}\xi+3\Theta_{yyx}\xi^2+\Theta_{yyy}\xi^3\big)\der z,\qquad
\tilde\omega^4 =\xi \der z+\der w,\\
\tilde\omega^5
 =\der y-\xi \der x-\big(\Theta_{xx}+2 \xi \Theta_{xy}+\xi^2 \Theta_{yy}\big)\der z.
\end{gather*}
Following \cite{annur}, we now pass to the new coordinates $(\tilde x,\tilde y,\tilde z,\tilde p,\tilde t)$ on ${\mathbb T}(M^4)$
\begin{gather*}
x \mapsto \tilde t, \qquad w \mapsto \tilde y \qquad z \mapsto \tilde x, \qquad -\xi \mapsto \tilde p, \qquad y \mapsto \tilde z-\tilde p \tilde t.
\end{gather*}
We obtain the distribution annihilated by the following 1-forms:
\begin{gather*}
\tilde \omega^3 =-\der \tilde p-\tilde A\der \tilde x,\qquad
\tilde\omega^4 =-\tilde p \der \tilde x+\der \tilde y,\\
\tilde\omega^5 =\der \tilde z-\tilde p \der \tilde t-\tilde t \der \tilde p+\tilde p \der \tilde t-\tilde B\der \tilde x
 =\der \tilde z-\tilde t \der \tilde p-\tilde B\der \tilde x
 =\der \tilde z+(\tilde t \tilde A-\tilde B) \der \tilde x,
\end{gather*}
where $\tilde A$ and $\tilde B$ are coordinate transforms of the functions
\begin{gather*}
A(w,x,y,z,\xi) =\Theta_{xxx}+3\Theta_{yxx}\xi+3\Theta_{yyx}\xi^2+\Theta_{yyy}\xi^3,\\
B(w,x,y,z,\xi) =\Theta_{xx}+2 \xi \Theta_{xy}+\xi^2 \Theta_{yy}
\end{gather*}
respectively.
This suggests taking
\begin{gather*}
\tilde A=-H'(t), \tilde B=-H(t)
\end{gather*}
to obtain the Legendre transformed 1-forms in Section~\ref{ltc}.
Passing back to coordinates $(w,x,y$, $z,\xi)$ on ${\mathbb T}(M^4)$, this gives
\begin{gather*}
-H'(x) =\Theta_{xxx}+3\Theta_{yxx}\xi+3\Theta_{yyx}\xi^2+\Theta_{yyy}\xi^3,\\
-H(x) =\Theta_{xx}+2 \xi \Theta_{xy}+\xi^2 \Theta_{yy},
\end{gather*}
so that
\begin{gather*}
\Theta(x)=-\iint H(x) \der x \der x
\end{gather*}
will satisfy the condition. We have $\Theta_{xx}=-H(x)$.
We have the following theorem.

\begin{Theorem}
The An--Nurowski twistor distribution $\mathcal{D}$ on the circle twistor bundle ${\mathbb T}(M^4) \to M^4$ of $(M^4,g)$ with the Pleba\'nski metric
\begin{gather*}
g=\der w \der x+\der z \der y+H(x) \der z^2
\end{gather*}
and the function $H(x)$ has split $G_2$ as its group of local symmetries
provided that $H(x)$ is one of the following up to fractional linear transformations in $s$:
\begin{enumerate}\itemsep=0pt
\item[$1.$]
The function $H(x)$ is given by
\begin{gather*}
H(x)=\iint \frac{\dot{s}^2}{s^{\frac{4}{3}}(s-1)^{\frac{4}{3}}} \der x \der x,
\end{gather*}
where $s(x)$ is a solution to the $3^{\text{rd}}$ order ODE \eqref{schwarzian}
\begin{gather*}
\{s,x\}+\frac{\dot{s}^2}{2}V(s)=0
\end{gather*}
with $(\alpha,\beta,\gamma)$ given by either
$\big(\frac{4}{3}, \frac{4}{3}, \frac{4}{3}\big)$ or $\big(\frac{4}{3}, \frac{1}{3},\frac{1}{3}\big)$.
\item[$2.$]
The function $H(x)$ is given by
\begin{gather*}
H(x)=\iint \frac{\dot{s}^2}{s^{\frac{4}{3}}(s-1)} \der x \der x,
\end{gather*}
where $s(x)$ is a solution to \eqref{schwarzian}
with $(\alpha,\beta,\gamma)$ one of $\big(\frac{2}{3}, \frac{1}{2}, \frac{1}{3}\big)$, $
\big(\frac{2}{3}, \frac{1}{2},\frac{4}{3}\big)$ or
$\big(\frac{2}{3}, 2,\frac{1}{3}\big)$.
\item[$3.$]
The function $H(x)$ is given by
\begin{gather*}
H(x)=\iint \frac{\dot{s}^2}{s^{\frac{5}{3}}(s-1)^{\frac{5}{3}}} \der x \der x,
\end{gather*}
where $s(x)$ is a solution to \eqref{schwarzian}
with $(\alpha,\beta,\gamma)$ either $\big(\frac{8}{3}, \frac{2}{3}, \frac{2}{3}\big)$ or $
\big(\frac{2}{3}, \frac{2}{3},\frac{2}{3}\big)$.

\item[$4.$]
The function $H(x)$ is given by $H(x)=x^m$ where $m \in \big\{{-}2,-\frac{1}{2},\frac{1}{2},2\big\}$.
\item[$5.$]
The function $H(x)$ is given by
\begin{gather*}
H(x)=-\frac{1}{192}\frac{\sqrt{x+C}(4x+3B+C)}{\sqrt{x+B}(B-C)^3}.
\end{gather*}
This corresponds to the solution obtained from \eqref{nsol}.

\item[$6.$] The function $H(x)$ is the Legendre transform of the function $F(q)$ with $q=H'(x)$ and
\begin{gather*}
H(x)=qx-F(q)=xH'(x)-F(H'(x)).
\end{gather*}
In this case $F(q)$ can be given by one of the following:
\begin{gather*}
(a) \quad F(q)=\iint \frac{\dot{s}^{\frac{3}{2}}}{s(s-1)}\der q \der q
\end{gather*}
where $s(q)$ is again a solution to the $3^{\text{rd}}$ order ODE \eqref{schwarzian} with $(\alpha,\beta,\gamma)$ one of $(3,3,3)$ or
$\big(3, \frac{1}{3},\frac{1}{3}\big)$.
\begin{gather*}
(b) \quad F(q)=\iint \frac{\dot{s}^{\frac{3}{2}}}{s(s-1)^{\frac{3}{4}}}\der q \der q,
\end{gather*}
where $s(q)$ is a solution to \eqref{schwarzian}
with $(\alpha,\beta,\gamma)$ one of
$\big(\frac{3}{2}, \frac{1}{3}, \frac{1}{2}\big)$, $
\big(\frac{3}{2}, 3,\frac{1}{2}\big)$ or
$\big(\frac{3}{2}, \frac{1}{3},\frac{9}{2}\big)$.
\begin{gather*}
(c) \quad F(q)=\iint \frac{\dot{s}^{\frac{3}{2}}}{s^\frac{5}{4}(s-1)^{\frac{5}{4}}}\der q \der q,
\end{gather*}
where $s(q)$ is a solution to \eqref{schwarzian}
with $(\alpha,\beta,\gamma)$ one of $\big(6,\frac{3}{2}, \frac{3}{2}\big)$ or $\big(\frac{2}{3}, \frac{3}{2},\frac{3}{2}\big)$.
\begin{gather*}
(d) \quad F(q)=q^m,
\end{gather*} where $m \in \big\{{-}1,\frac{1}{3},\frac{2}{3},2\big\}$.
\begin{gather*}
(e) \quad F(q)=-\frac{1}{6}\frac{(q+B)^{\frac{1}{3}}(q+C)^{\frac{2}{3}}}{(B-C)^2},
\end{gather*}
again corresponding to the solution~\eqref{nsol}.
\end{enumerate}
\end{Theorem}

\subsection*{Acknowledgements}
This work is inspired by the paper of~\cite{annur}. The author would like to thank Daniel An, Pawe{\l} Nurowski, Travis Willse and the anonymous referees for comments. Part of this work is supported by the Grant agency of the Czech Republic P201/12/G028.

\pdfbookmark[1]{References}{ref}
\LastPageEnding

\end{document}